\documentclass[12pt,a4paper]{amsart}
\usepackage{amsfonts}
\usepackage{mathrsfs}
\usepackage{amsthm,amsxtra}
\usepackage{amssymb}
\usepackage{amsmath}
\usepackage{amscd}
\usepackage{hyperref}
\hypersetup{
colorlinks=true,
linkcolor=blue,
citecolor=gray,
}
\usepackage[utf8]{inputenc}
\usepackage{t1enc}
\usepackage[mathscr]{eucal}
\usepackage{indentfirst}
\usepackage{graphicx}
\graphicspath{ {images/} }
\usepackage{subfig}
\usepackage{enumerate}
\usepackage{graphics}
\usepackage{pict2e}
\usepackage{epic}
\numberwithin{equation}{section}
\usepackage[margin=2.9cm]{geometry}
\usepackage{epstopdf} 
\usepackage{longtable}
\usepackage[inline]{enumitem}
\usepackage{tikz-cd}
\usepackage{tikz}

\usepackage{todonotes}

\DeclareMathOperator{\MCG}{MCG}

\DeclareMathOperator{\Out}{Out}
\DeclareMathOperator{\Aut}{Aut}

\newcommand{\Pants}{\mathcal{P}} % for free bases graphs
\newcommand{\FS}{\mathcal{FS}} % for free bases graphs
\newcommand{\Curve}{\mathcal{C}} % for curve complexes

\theoremstyle{plain}
\newtheorem{theorem}{Theorem}[section]
\newtheorem{proposition}[theorem]{Proposition}
\newtheorem{lemma}[theorem]{Lemma}
\newtheorem{corollary}[theorem]{Corollary}

\theoremstyle{definition}
\newtheorem{definition}[theorem]{Definition}

\theoremstyle{remark}
\newtheorem{remark}[theorem]{Remark}
\newtheorem*{remark*}{Remark}

\begin{document}

\title{The pants graph of a free group}
	
\author{Donggyun Seo}
\address{Institute of Mathematical Sciences \\ Chungnam National University \\ Daejeon, South Korea}
\email{seodonggyun@cnu.ac.kr \\ seodonggyun7@gmail.com}
\urladdr{https://www.seodonggyun.com}


\begin{abstract}
    We introduce the concept of a pants decomposition for a finitely generated free group and construct the corresponding pants graph. 
    A pants decomposition of a free group leads to the formation of a simplicial graph, referred to as the pants graph of a free group, consisting of all possible pants decompositions.
    The natural isometric action of the outer automorphism group of the free group on the pants graph induces a coarsely surjective orbit map.
    Additionally, we construct a coarsely Lipschitz map from the pants graph to the free splitting complex. 
    These results imply that the pants graph of a free group is both connected and unbounded.
\end{abstract}

\maketitle

% \textcolor{red}{A comment on the title: When a construction is said to arise from the rose, I assume it has something to do with the rose, as distinct from other marked graphs; or to speak more algebraically I’d assume it has something to do with a free basis of $F_n$. In this paper, the constructions all just proceed directly from $F_n$, so I personally think it would be more descriptive to say $\mathcal{P}_n$ is the pants graph of $F_n$ rather than of a rose. This is just my opinion.}

% \textcolor{blue}{I agree with the referee's opinion. I replace ``the pants graph of a rose'' into ``the pants graph of a free group'' and convert the explanations fitted into the situations.}

% \textcolor{blue}{We have replaced the phrase ``the outer space'' with ``Outer space'' throughout the manuscript.}

\section{Introduction}

For a free group \( F_n \) of rank \( n \), the dynamical behavior of its outer automorphism group \( \Out(F_n) \) has been extensively studied through various isometric actions. One particularly well-known object in this context is the Culler--Vogtmann Outer space, which is equipped with a Lipschitz metric. However, the Lipschitz metric is not symmetric, and its symmetrization, while intriguing, is not a geodesic metric, as shown by Francaviglia and Martino \cite{MR2839451}.

In a parallel development, the Thurston metric on Teichm\"uller space, closely related to the Lipschitz metric, has been studied extensively. The Thurston metric diverges from the Teichm\"uller metric in the thin parts, as demonstrated by Choi and Rafi \cite{MR2377122}, yet remains uniformly bounded in a thick part.

% \textcolor{red}{\small Page 1, third paragraph: I think you should distinguish between who introduced these complexes and who proved they are hyperbolic, to avoid potential confusion about attributions. Also, you should comment that the sphere complex is the same thing as the free splitting complex.}

% \textcolor{blue}{I have removed the discussion of Gromov hyperbolicity, as it is not directly related to the subject of this paper. In its place, I have reorganized the relevant sections to improve coherence and flow.}

Over the last decade, researchers have developed several simplicial graphs associated with \( \Out(F_n) \), such as the free splitting graphs \cite{MR1314940}, the cyclic splitting graphs \cite{MR3275303}, the free factor graphs \cite{MR1660045}, the intersection graphs \cite{MR2496058}, and others. Coarsely Lipschitz maps from Outer space to these complexes have been constructed, providing insights into the geometric structure of \( \Out(F_n) \) \cite{MR3231221, MR3653318}. A map between two metric spaces is said to be coarsely Lipschitz if it satisfies a specific bounded distortion condition between distances in the spaces.

% \textcolor{blue}{We have explicitly stated the coarsely Lipschitz condition in Definition~1.1 and adopted the notion of a coarsely surjective map. In addition, we have introduced a new relation between two metric spaces induced by a coarsely surjective, coarsely Lipschitz map.}

\begin{definition}
    Let \( \varphi: \mathcal{X} \to \mathcal{Y} \) be a map between two metric spaces \( (\mathcal{X}, d_\mathcal{X}) \) and \( (\mathcal{Y}, d_\mathcal{Y}) \).
    \begin{enumerate}
        \item For $A \geq 1$ and $B \geq 0$, we say \( \varphi \) is \emph{\( (A, B) \)--coarsely Lipschitz} if
        \[
            d_\mathcal{Y}(\varphi(x), \varphi(y)) \leq A d_\mathcal{X}(x, y) + B
        \]
        for every $x, y \in \mathcal{X}$.
        \item For \( C \geq 0 \), we say \( \varphi \) is \emph{\( C \)--coarsely surjective} if \( \mathcal{Y} \) is the \( C \)--neighborhood of the image \( \varphi(\mathcal{X}) \).
    \end{enumerate}
    We write \( \mathcal{X} \succ \mathcal{Y} \) if there exists an \( C \)--coarsely surjective and \( (A, B) \)--coarsely Lipschitz map from \( \mathcal{X} \) to \( \mathcal{Y} \) for some \( A \), \( B \) and \( C \).
\end{definition}

For orientable surfaces of finite type, Brock \cite{MR1969203} demonstrated that the pants graph representing the simplicial structure of all pants decompositions captures the large-scale geometry of the Weil--Petersson metric on Teichm\"uller space. This relationship highlights how combinatorial properties of the pants graph can inform geometric structures.

In this paper, we focus on pants decompositions of surfaces marked by a free group, which may be orientable or non-orientable. We introduce four elementary moves that modify these decompositions and define a simplicial graph, denoted by \( \Pants_n \), as the pants graph of a free group. Additionally, we study the induced subgraph \( \Pants^+_n \), which consists of pants decompositions of orientable surfaces marked by a free group.

Our main result demonstrates the relationship between the pants graph and well-known \( \Out(F_n) \)-complexes, as shown in the following theorem:

\begin{theorem} \label{thm:main}
    For each \( n \geq 3 \), the following sequence holds:
    \[
        \mathcal{K}_n \succ \Pants^+_n \succ \Pants_n \succ \FS_n
    \]
    where \( \mathcal{K}_n \) is the spine of Outer space and \( \FS_n \) is the free splitting complex.
    In addition, $\Pants_2^+$ is quasi-isometric to the Farey graph.
\end{theorem}

% \textcolor{red}{\textbf{Page 2, paragraph after Theorem 1.2:} I don’t follow the sentence ``This implies both $\mathcal{P}^+_n$ and $\mathcal{P}_n$ are connected.'' Receiving a coarsely Lipschitz coarsely surjective map from a connected space should only guarantee coarse connectivity.}

% \textcolor{blue}{The observation is that coarse connectivity implies connectivity. I added the detail in the paragraph.}

Notice \( \mathcal{K}_n \) is connected \cite[Proposition 3.1.2]{MR830040} and \( \Pants_n \) and \( \Pants_n^+\) are endowed with the combinatorial metrics.
This implies both \( \Pants^+_n \) and \( \Pants_n \) are connected.
Indeed, given any two vertices $x, y \in \Pants_n^+$, if $\varphi \colon \mathcal{K}_n \to \Pants_n^+$ is a $C$--coarsely surjective $(A, B)$--coarsely Lipschitz map, there exist $z_x, z_y \in \mathcal{K}_n$ such that $d_\Pants(x, \varphi(z_x)) \leq C$ and $d_\Pants(y, \varphi(z_y)) \leq C$. We obtain
\begin{align*}
    d_\Pants(x, y) &\leq d_\Pants(x, \varphi(z_x)) + d_\Pants(\varphi(z_x), \varphi(z_y)) + d_\Pants(\varphi(z_y), y) \\
    &\leq A\, d_\mathcal{K}(z_x, z_y) + B + 2C \;<\; \infty,
\end{align*}
which guarantees the existence of a path joining $x$ and $y$.
Meanwhile, the free splitting complex \( \FS_n \) is unbounded, guaranteed by the existence of loxodromic elements \cite{MR3073931, MR4009387}. It follows that \( \Pants^+_n \) and \( \Pants_n \) are also unbounded. This leads to the following corollary:

\begin{corollary} \label{cor:conn_ubdd}
    For each \( n \geq 3 \), both \( \Pants^+_n \) and \( \Pants_n \) are connected and unbounded.
\end{corollary}

\subsection{Questions}

Our study of pants graphs of free groups is only in its early stages, and many aspects remain unknown. In this section we collect several remarks and questions concerning pants graphs of free groups, presented in comparison with the existing results on pants graphs of surfaces.

The pants graph of a compact surface becomes simply connected once polygons corresponding to finitely many relations are attached, as shown by Hatcher and Thurston \cite{MR579573} in the orientable case and by Papadopoulos and Penner \cite{MR3460762} in the nonorientable case. This property allows the pants graph of a surface to serve as a deformation space for the mapping class group. By analogy, we expect that the pants graph of a free group may play a similar role; in particular, we conjecture that it becomes simply connected after attaching polygons arising from finitely many relations.

For any homotopy equivalence between a rose and a surface, the induced map from the pants graph of the surface to the pants graph of the free group is coarsely Lipschitz. We expect, however, that this map is in fact a quasi--isometric embedding. Regardless of the outcome, resolving this question will clarify the relationship between pants graphs of surfaces and pants graphs of free groups.

In connection with Corollary \ref{cor:conn_ubdd}, we note the following. In this article, we are unable to establish any meaningful relationship between the Culler--Vogtmann Outer space and the pants graph. By analogy with the well-known quasi-isometry between the Weil--Petersson metric and the pants graph of a surface, one might naturally expect a similar connection between Outer space and the pants graph. However, work by Aougab, Clay, and Rieck \cite{MR4544143} shows that there is no analogue of a Weil--Petersson-type metric on Outer space. In particular, the metric completion of Outer space with respect to any entropy or pressure metric necessarily admits a global fixed point.

In contrast, \( \Pants_n \) cannot admit a global fixed point, since it contains outer automorphisms with unbounded orbits. This follows from the fact that the coarsely Lipschitz maps we constructed are ``coarsely \( \Out(F_n) \)--equivariant'' and that there exists an element acting loxodromically on \( \FS_n \). Thus, while any metric completion of Outer space is forced into a fixed-point phenomenon, the pants graph resists such a structure. This suggests a fundamental distinction between the two settings, though at present the precise reason for this difference remains unclear.

Finally, we turn to the question of determining the automorphism group of the pants graph of a free group. By a result of Margalit \cite{MR2040283}, the automorphism group of the pants graph of an orientable surface is virtually isomorphic to the extended mapping class group. More recently, Stukow and Szepietowski \cite{2025arXiv250712613S} showed that the same holds for nonorientable surfaces, where the automorphism group of the pants graph is virtually isomorphic to the mapping class group. This naturally leads to the question of whether the automorphism group of the pants graph of a free group is virtually isomorphic to the outer automorphism group \(\Out(F_n)\).

\subsection{Acknowledgement}
I would like to express my deep gratitude to Sang-hyun Kim, my PhD advisor, for his invaluable discussions and unwavering support throughout this project. His contributions were essential to the completion of this work.
This project was initiated at the conference ``1,2,3: Curves, Surfaces, and 3-Manifolds,'' held in 2023 at Technion, Israel. The minicourses by Mladen Bestvina and Camille Horbez provided the primary motivation for this work. I am especially grateful for the insightful comments from Mladen Bestvina, Karen Vogtmann, and Camille Horbez, which greatly influenced the direction of this research.
I would also like to thank Inhyeok Choi, Juhun Baik, and Sanghoon Kwak for their careful attention to this project and for providing valuable feedback.
This work was supported by the National Research Foundation of Korea(NRF) grant funded by the Korea government(MSIT) 2021R1C1C2005938.

\section{The Orientable Pants Graph for the Rank-Two Free Group}

% \textcolor{red}{\textbf{Page 3, Section 2 title:} The title of this section still has the old name ``pants
% graph of a rose''.}

% \textcolor{blue}{Referring to the referee's comments, we fix the section title.}

Before introducing the pants graph of a free group of arbitrary rank, we first examine the \emph{orientable pants graph} of the free group \( F_2 \) of rank two. Fix a rose $R_2$ with two petals, and identify \( F_2 \) with some fundamental group of \( R_2 \). A compact orientable surface \(X\) is homotopy equivalent to \(R_2\) if and only if \(X\) is either a pair of pants (also called a \emph{trouser}) or a once-holed torus.  
Originally, a \emph{pants decomposition} of a surface is defined as a maximal collection of homotopically non-trivial, pairwise disjoint, non-peripheral essential simple closed curves.

The pants graph of a once-punctured torus coincides with its curve graph, known as the \emph{Farey graph}, which is illustrated on the left side of Figure~\ref{fig:Farey}.  
In contrast, the pants graph of a pair of pants is trivial--it consists of a single vertex, represented by the empty set, since there are no non-peripheral essential simple closed curves on it.

For an orientable surface \(X\), a \emph{marking} is defined as a homotopy equivalence \(h_X: R_2 \to X\), and the pair \((h_X, X)\) is called a \emph{marked surface}.  
A pants decomposition of a surface can thus be redefined as a maximal collection of homotopically disjoint essential simple closed curves.  
In the case of a marked surface, the decomposition is denoted by the triple \((P_X, h_X, X)\), which we refer to as a \emph{pants decomposition of \(R_2\)}.

Two marked surfaces \((h_X, X)\) and \((h_Y, Y)\) are called \emph{marking-equivalent} if there exists a homeomorphism \(f: X \to Y\) such that \(h_Y\) is homotopic to \(f \circ h_X\).  
In this case, the homeomorphism \(f\) is referred to as a \emph{marking equivalence}.  
Two pants decompositions are said to be \emph{equivalent} if there exists a marking equivalence between them, such that one decomposition is the image of the other under this equivalence.

% \textcolor{blue}{We have added the fact that all marked once--holed tori are marking--equivalent, and placed this remark before the definition of the orientable pants graph of the rank--two free group.}

We observe that all marked once-holed tori are marking-equivalent.
Let \((h_X, X)\) and \((h_Y, Y)\) be two marked once-holed tori.
Note that there exists a deformation retract \(r_X : X \to R_2\) which serves as a homotopy inverse to \(h_X\).
For any simple closed curve \(\gamma \subset X\), the image \((h_Y \circ r_X)(\gamma)\) is homotopic to a simple closed curve in \(Y\).
Hence, the map \(h_Y \circ r_X\) induces an isometry from the Farey graph of \(X\) to that of \(Y\).
By Luo \cite{MR1722024} and Margalit \cite{MR2040283}, this in turn yields a homeomorphism \(f : X \to Y\).
Therefore, we conclude 
\[
    f \circ h_X \simeq h_Y \circ r_X \circ h_X \simeq h_Y \circ \mathrm{id}_{R_2} \simeq h_Y,
\] 
which establishes the claimed equivalence.

We now define a simplicial graph, denoted \(\Pants_2^+\), as follows.
The vertices of \(\Pants_2^+\) are equivalence classes of pants decompositions of \(R_2\).
An edge between two vertices \(P\) and \(Q\) is defined by the following conditions:
\begin{enumerate}
    \item \(P\) corresponds to a pants decomposition \((P_X, h_X, X)\), where \(X\) is a once-holed torus.
    \item \(Q\) corresponds to a trivial decomposition \((\emptyset, h_Y, Y)\), where \(Y\) is a pair of pants.
    \item There exists a map \(c: S^1 \to R_2\) such that \(h_X \circ c\) is homotopic to the simple closed curve \(P_X\) and \(h_Y \circ c\) is homotopic to a boundary component of \(Y\).
\end{enumerate}
We call \(\Pants_2^+\) the \emph{orientable pants graph of \(F_2\)}.

% \textcolor{red}{Page 3, paragraph before Theorem 2.1: A little more detail/explanation would be good here. In particular, you should comment that all marked once punctured tori are marking equivalent.}

% \textcolor{blue}{We have added more details by including two explanatory remarks. First, we note that all marked once--holed tori are marking--equivalent. Second, we explain that there exists a canonical homotopy equivalence between adjacent vertices.}

For the consistency of the definition of pants graphs in the later section, we introduce an equivalent description of the edges of \(\Pants_2^+\).
This description is given by a homotopy connecting two pants decompositions.
Given a pants decomposition \((P_X, h_X, X)\) where \(X\) is a once-holed torus, we modify \(X\) by a homotopy to obtain a pair of pants as follows.
More precisely, retract \(X\) onto a figure-eight while preserving the simple closed curve \(P_X\) as a circle, and then construct a pair of pants \((\emptyset, h_Y, Y)\) in which this figure-eight is embedded.
We then join the vertices represented by \((P_X, h_X, X)\) and \((\emptyset, h_Y, Y)\) by an edge.

If a vertex of $\Pants_2^+$ corresponds to a simple closed curve on a once-holed torus, then it is adjacent to infinitely many vertices in the pants graph.
On the other hand, if a vertex corresponds to a pair of pants, it is adjacent to exactly three vertices, i.e., finitely many. 
These three vertices are adjacent to each other in the Farey graph because any two of them form a free basis of the fundamental group of the once-holed torus.

\begin{figure}
    \centering
    \begin{tikzcd}
        % \begin{tikzpicture}[black, very thin]
        %     % \draw (0, 0) circle (3);
        %     \draw (-3, 0) -- (3, 0);
        %     \foreach \i in {2, ..., 12}
        %         \pgfmathsetmacro{\k}{2^\i}
        %         \foreach \j in {1, ..., \k}
        %             \draw ({(360/\k)*\j}:3) arc ({(360/\k)*\j+90}:{(360/\k)*\j+90+(180-360/\k)}:{3*tan(180/\k)});
        % \end{tikzpicture}
        \includegraphics[]{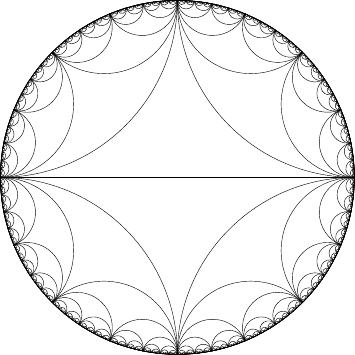}
        &
        % \begin{tikzpicture}[black, very thin]
        %     % \draw (0, 0) circle (3);
        %     \foreach \i in {1, ..., 12}{
        %         \pgfmathsetmacro{\k}{2^\i}
        %         \foreach \j in {1, ..., \k}{
        %             \draw ({(360/\k)*(\j-1)+180/\k}:{3-5/(2^\i)}) -- ({(360/\k)*(\j-1)}:3);
        %             \draw ({(360/\k)*(\j-1)+180/\k}:{3-5/(2^\i)}) -- ({(360/\k)*(\j-1)+180/\k}:3);
        %             \draw ({(360/\k)*(\j-1)+180/\k}:{3-5/(2^\i)}) -- ({(360/\k)*(\j-1)+360/\k}:3);
        %         }
        %     }
        % \end{tikzpicture}
        \includegraphics[]{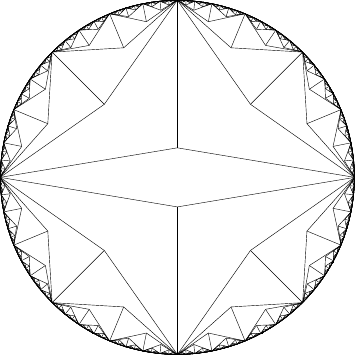}
    \end{tikzcd}
    \caption{the Farey graph and the orientable pants graph of $F_2$}
    \label{fig:Farey}
\end{figure}

\begin{theorem} \label{thm:rank_2_hyperbolicity}
    $\Pants_2^+$ is quasi-isometric to the Farey graph.
\end{theorem}

\begin{proof}
    % \textcolor{red}{Page 3, first sentence of proof of Theorem 2.1: I don’t think $\mathcal{C}(S_{1,1})$ is homeomorphic to the Euclidean plane. (That’s only true if you remove the vertices.)}

    % \textcolor{blue}{I have deleted this line.}

    Let $S_{1,1}$ be a once-holed torus with a marking $h: R_n \to S_{1,1}$.  
    Consider the curve complex $\Curve(S_{1,1})$, which is the flag complex of the Farey graph.
    Let $\hat\Curve(S_{1,1})$ denote the barycentric subdivision of the curve complex $\Curve(S_{1,1})$.  
    Notice that the 1-skeleton of $\hat\Curve(S_{1,1})$ is quasi-isometric to $\Curve(S_{1,1})$.

    Now, define a map $\varphi: \Pants_2^+ \to \hat\Curve(S_{1,1})^{(1)}$ as follows:
    \begin{itemize}
        \item For each pants decomposition $P$ of $S_{1,1}$, let $\varphi(P)$ be $P$ itself.  
        \item For a pair of pants $Q$, define $\varphi(Q)$ as the midpoint of the vertices representing the boundary curves of $Q$.  
    \end{itemize}
    This map $\varphi$ preserves the adjacency of vertices, and we extend it to a simplicial map.  

    The image of $\varphi$ contains all vertices except for the midpoints of edges of $\mathcal{C}(S_{1,1})$.  
    Observe that the closed 1-neighborhood of the image of $\varphi$ covers $\hat\Curve(S_{1,1})^{(1)}$.  
    Thus, the image of $\varphi$ is 1-dense.  

    Now, consider two vertices $P', P'' \in \Pants_2^+$.  
    Let $\varphi(P') = \gamma_0, \gamma_1, \dots, \gamma_k = \varphi(P'')$ be a geodesic in $\hat\Curve(S_{1,1})$.  
    If some $\gamma_i$ is not contained in the image of $\varphi$, then it must be the midpoint of an edge.  
    This implies that both $\gamma_{i-1}$ and $\gamma_{i+1}$ belong to the image of $\varphi$.  

    Therefore, there exists a vertex $\gamma_i'$ in the image of $\varphi$ such that the sequence $\gamma_{i-1}, \gamma_i', \gamma_{i+1}$ forms a path.  
    By applying this method inductively, we obtain a geodesic path from $\varphi(P')$ to $\varphi(P'')$ that lies entirely within the image of $\varphi$.  

    This shows that $\varphi$ is an isometric embedding into $\hat\Curve(S_{1,1})$.  
    Hence, $\Pants_2^+$ is quasi-isometric to $\hat\Curve(S_{1,1})$.  
    Since $\hat\Curve(S_{1,1})$ is quasi-isometric to $\Curve(S_{1,1})$, it follows that $\Pants_2^+$ is quasi-isometric to $\Curve(S_{1,1})$.
\end{proof}

\section{Pants Decompositions and Pants Graphs of Free Groups}

In this section, we introduce four elementary moves that transform pants decompositions, and using these moves, we define the pants graph of the free group of rank $n$.
We now fix $n \geq 3$.
Let us identify the free group $F_n$ with a fundamental group of the $n$-rose $R_n$.

To define a pants decomposition of $F_n$, we need compact surfaces that are homotopy equivalent to $R_n$.  
An orientable surface of genus $g$ with $b$ boundary curves is homotopy equivalent to the $n$-rose if and only if $n = 2g - 1 + b$.  
For the non-orientable case, a non-orientable surface with $c$ crosscaps and $b$ boundary curves is homotopy equivalent to the $n$-rose if and only if $n = c - 1 + b$.

\subsection{Pants decomposition}

% \textcolor{red}{Page 4, Section 3.1: For clarity, I think it would be best to emphasize that X is a compact surface with boundary.}

% \textcolor{blue}{I put it.}

Let $R_n$ denote the $n$-rose.
We define a \emph{marked pants decomposition} as a triple $(P_X, h_X, X)$, where:
\begin{enumerate}
    \item $X$ is a compact surface with boundary,
    \item $h_X: R_n \to X$ is a homotopy equivalence, and
    \item $P_X$ is a pants decomposition of $X$.
\end{enumerate}
Two marked pants decompositions $(P_X, h_X, X)$ and $(P_Y, h_Y, Y)$ are said to be \emph{equivalent} if there exists a homeomorphism $f: X \to Y$ such that:  
\begin{enumerate}
    \item $f(P_X) = P_Y$, and
    \item $f \circ h_X$ is homotopic to $h_Y$.
\end{enumerate}
Let us write $[P_X, h_X, X]$ for the equivalence class of a marked pants decomposition $(P_X, h_X, X)$.  
An equivalence class of a marked pants decomposition is called a \emph{pants decomposition of $F_n$}.

\subsection{Elementary moves}

We now introduce elementary moves that modify one pants decomposition into another.  
For a compact orientable surface, Hatcher and Thurston \cite{MR579573} defined two elementary moves among pants decompositions of the surface, which they used to find a presentation of the mapping class group.  
Papadopoulos and Penner \cite{MR3460762} extended this framework, providing four elementary moves for pants decompositions of non-orientable surfaces.

\subsubsection{Elementary move I}

Elementary move I replaces a simple closed curve in a pants decomposition with another curve, such that the regular neighborhood of the union of the original and new curve forms a four-holed sphere.  
More precisely, for two pants decompositions $P$ and $P'$, we say that \emph{$P'$ is obtained from $P$ by Elementary move I} if:
\begin{enumerate}
    \item $P = [P_X, h_X, X]$ and $P' = [P_X', h_X, X]$ for some homotopy equivalence $h_X$, and
    \item there exist simple closed curves $\gamma$ and $\gamma'$ such that $P_X - \gamma = P_X' - \gamma'$ and a regular neighborhood of $\gamma \cup \gamma'$ is a four-holed sphere.
\end{enumerate}
See Figure \ref{fig:switching} for an illustration.

\begin{figure}[ht]
    \centering
    \begin{tikzcd}
    \begin{tikzpicture}[thick]
        % \draw[very thin] (-3,-3) grid (3,3);
        \draw (-2,-1) arc (-90:90:1);
        \draw (-1,2) arc (-180:0:1);
        \draw (2,1) arc (90:270:1);
        \draw (1,-2) arc (0:180:1);
        
        \draw[dashed, rotate around={45:(1,1)}] (1,{sqrt(2)}) arc (90:270:{0.3*(sqrt(2)-1)} and {sqrt(2)-1});
        \draw[rotate around={45:(1,1)}] (1,{2-sqrt(2)}) arc (-90:90:{0.3*(sqrt(2)-1)} and {sqrt(2)-1});
        \draw[dashed, rotate around={135:(1,-1)}] (1,{-2+sqrt(2)}) arc (90:270:{0.3*(sqrt(2)-1)} and {sqrt(2)-1});
        \draw[rotate around={135:(1,-1)}] (1,{-sqrt(2)}) arc (-90:90:{0.3*(sqrt(2)-1)} and {sqrt(2)-1});
        \draw[rotate around={225:(-1,-1)}] (-1,{-2+sqrt(2)}) arc (90:270:{0.3*(sqrt(2)-1)} and {sqrt(2)-1});
        \draw[dashed, rotate around={225:(-1,-1)}] (-1,{-sqrt(2)}) arc (-90:90:{0.3*(sqrt(2)-1)} and {sqrt(2)-1});
        \draw[rotate around={315:(-1,1)}] (-1,{sqrt(2)}) arc (90:270:{0.3*(sqrt(2)-1)} and {sqrt(2)-1});
        \draw[dashed, rotate around={315:(-1,1)}] (-1,{2-sqrt(2)}) arc (-90:90:{0.3*(sqrt(2)-1)} and {sqrt(2)-1});

        \draw (1,0) arc (0:180:1 and 0.3);
        \draw[dashed] (-1,0) arc (180:360:1 and 0.3);
    \end{tikzpicture} \ar[r, leftrightarrow, "\mathrm{I}"]
    &
    \begin{tikzpicture}[thick]
        % \draw[very thin] (-3,-3) grid (3,3);
        \draw (-2,-1) arc (-90:90:1);
        \draw (-1,2) arc (-180:0:1);
        \draw (2,1) arc (90:270:1);
        \draw (1,-2) arc (0:180:1);
        
        \draw[dashed, rotate around={45:(1,1)}] (1,{sqrt(2)}) arc (90:270:{0.3*(sqrt(2)-1)} and {sqrt(2)-1});
        \draw[rotate around={45:(1,1)}] (1,{2-sqrt(2)}) arc (-90:90:{0.3*(sqrt(2)-1)} and {sqrt(2)-1});
        \draw[dashed, rotate around={135:(1,-1)}] (1,{-2+sqrt(2)}) arc (90:270:{0.3*(sqrt(2)-1)} and {sqrt(2)-1});
        \draw[rotate around={135:(1,-1)}] (1,{-sqrt(2)}) arc (-90:90:{0.3*(sqrt(2)-1)} and {sqrt(2)-1});
        \draw[rotate around={225:(-1,-1)}] (-1,{-2+sqrt(2)}) arc (90:270:{0.3*(sqrt(2)-1)} and {sqrt(2)-1});
        \draw[dashed, rotate around={225:(-1,-1)}] (-1,{-sqrt(2)}) arc (-90:90:{0.3*(sqrt(2)-1)} and {sqrt(2)-1});
        \draw[rotate around={315:(-1,1)}] (-1,{sqrt(2)}) arc (90:270:{0.3*(sqrt(2)-1)} and {sqrt(2)-1});
        \draw[dashed, rotate around={315:(-1,1)}] (-1,{2-sqrt(2)}) arc (-90:90:{0.3*(sqrt(2)-1)} and {sqrt(2)-1});

        \draw (0,-1) arc (-90:90:0.3 and 1);
        \draw[dashed] (0,1) arc (90:270:0.3 and 1);
    \end{tikzpicture}
\end{tikzcd}
    \caption{Elementary move I}
    \label{fig:switching}
\end{figure}

\subsubsection{Elementary move II}

The key characteristic of Elementary move II is that it involves a modification of the surface itself. 
This move provides a homotopy equivalence $\eta: X \to Y$, which transforms a pants decomposition $[P_X, h_X, X]$ into a new pants decomposition $[P_Y, \eta \circ h_X, Y]$ for some marked pants decomposition $P_Y$, as illustrated in Figure \ref{fig:tying}.

To apply Elementary move II, the pants decomposition $[P_X, h_X, X]$ must satisfy one of the following conditions:
\begin{enumerate}
    \item $X$ is either a four-holed sphere or a twice-holed torus.
    \item There exists a separating curve $\gamma \subset P_X$ that bounds either a four-holed sphere or a twice-holed torus.
\end{enumerate}

If the curve $\gamma$ bounds a four-holed sphere $A$, the homotopy equivalence $\eta: X \to Y$ is constructed as follows:
Since $A$ consists of two pairs of pants in $P_X$, one of these pants is bounded by two boundary curves of $X$.
By using a deformation retract, we deform this pair of pants into a $2$-rose, as shown in the lower left of Figure \ref{fig:tying}.
This transformation results in a wedge sum of the surface and a circle.
Within the pair of pants in the lower left of Figure \ref{fig:tying}, there is an arc that joins the two boundary curves.
We fold this arc along with part of the circle, producing a wedge sum of the surface and an arc, as depicted in the lower right of Figure \ref{fig:tying}.
The configuration in the lower right of Figure \ref{fig:tying} is a deformation retract of the upper right configuration.
Thus, there is a homotopy equivalence from the upper left to the upper right of Figure \ref{fig:tying}.

\begin{figure}[ht]
    \centering
    \begin{tikzcd}
    \begin{tikzpicture}[thick]
        \draw ({-2*sqrt(3)/2+cos(-60)},{1+sin(-60)}) arc (-60:0:1);
        \draw ({2*sqrt(3)/2+cos(180)},{1+sin(180)}) arc (180:300:1);
        \draw ({cos(-20)},{-2+sin(-20)}) arc (-20:120:1);
        \draw ({2*sqrt(3)+cos(120)},{-2+sin(120)}) arc (120:{180+20}:1);
        
        \draw (0,1) ellipse ({sqrt(3)-1} and 0.2);
        \draw[rotate around={120:({cos(210)},{sin(210)})}] ({cos(210)},{sin(210)}) ellipse ({sqrt(3)-1} and 0.2);
        
        \draw[rotate around={240:({cos(330)},{sin(330)})}] ({cos(330)+cos(0)*(sqrt(3)-1)},{sin(330)+sin(0)*(sqrt(3)-1)}) arc (0:180:{sqrt(3)-1} and 0.2);
        \draw[dashed, rotate around={240:({cos(330)},{sin(330)})}] ({cos(330)+cos(180)*(sqrt(3)-1)},{sin(330)+sin(180)*(sqrt(3)-1)}) arc (180:360:{sqrt(3)-1} and 0.2);
        
        \draw[rotate around={-60:({sqrt(3)+cos(30)},{-1+sin(30)})}] ({sqrt(3)+cos(30)},{-1+sin(30)}) ellipse ({sqrt(3)-1} and 0.2);
        
        \draw[dashed] ({sqrt(3)+cos(270)+cos(0)*(sqrt(3)-1)},{-1+sin(270)+sin(0)*(sqrt(3)-1)}) arc (0:180:{sqrt(3)-1} and 0.2);
        \draw ({sqrt(3)+cos(270)+cos(180)*(sqrt(3)-1)},{-1+sin(270)+sin(180)*(sqrt(3)-1)}) arc (180:360:{sqrt(3)-1} and 0.2);
    \end{tikzpicture}
    \ar[r,leftrightarrow, "\mathrm{II}"] \ar[d,dotted,leftrightarrow]
    &
    \begin{tikzpicture}[thick]
        \draw (0,0) circle (0.5);
        \draw ({sqrt(3)-1+(0.5+(sqrt(3)-1))*cos(-atan(2/(2/((sqrt(3)-1)+0.5)-((sqrt(3)-1)+0.5)/2)))},{(0.5+(sqrt(3)-1))*sin(-atan(2/(2/((sqrt(3)-1)+0.5)-((sqrt(3)-1)+0.5)/2)))}) arc ({-atan(2/(2/((sqrt(3)-1)+0.5)-((sqrt(3)-1)+0.5)/2))}:{atan(2/(2/((sqrt(3)-1)+0.5)-((sqrt(3)-1)+0.5)/2))}:{0.5+(sqrt(3)-1)});
        \draw ({-(sqrt(3)-1)+(0.5+(sqrt(3)-1))*cos(180-atan(2/(2/((sqrt(3)-1)+0.5)-((sqrt(3)-1)+0.5)/2)))},{(0.5+(sqrt(3)-1))*sin(180-atan(2/(2/((sqrt(3)-1)+0.5)-((sqrt(3)-1)+0.5)/2)))}) arc ({180-atan(2/(2/((sqrt(3)-1)+0.5)-((sqrt(3)-1)+0.5)/2))}:{180+atan(2/(2/((sqrt(3)-1)+0.5)-((sqrt(3)-1)+0.5)/2))}:{0.5+(sqrt(3)-1)});

        \draw ({(sqrt(3)-1)+2/((sqrt(3)-1)+0.5)-((sqrt(3)-1)+0.5)/2+(2/((sqrt(3)-1)+0.5)-((sqrt(3)-1)+0.5)/2)*cos(180-atan(2/(2/((sqrt(3)-1)+0.5)-((sqrt(3)-1)+0.5)/2)))},{-2+(2/((sqrt(3)-1)+0.5)-((sqrt(3)-1)+0.5)/2)*sin(180-atan(2/(2/((sqrt(3)-1)+0.5)-((sqrt(3)-1)+0.5)/2)))}) arc ({180-atan(2/(2/((sqrt(3)-1)+0.5)-((sqrt(3)-1)+0.5)/2))}:200:{2/((sqrt(3)-1)+0.5)-((sqrt(3)-1)+0.5)/2});
        \draw ({-((sqrt(3)-1)+2/((sqrt(3)-1)+0.5)-((sqrt(3)-1)+0.5)/2)+(2/((sqrt(3)-1)+0.5)-((sqrt(3)-1)+0.5)/2)*cos(-20)},{-2+(2/((sqrt(3)-1)+0.5)-((sqrt(3)-1)+0.5)/2)*sin(-20)}) arc (-20:{atan(2/(2/((sqrt(3)-1)+0.5)-((sqrt(3)-1)+0.5)/2))}:{2/((sqrt(3)-1)+0.5)-((sqrt(3)-1)+0.5)/2});
        \draw ({(sqrt(3)-1)+2/((sqrt(3)-1)+0.5)-((sqrt(3)-1)+0.5)/2+(2/((sqrt(3)-1)+0.5)-((sqrt(3)-1)+0.5)/2)*cos(180)},{2+(2/((sqrt(3)-1)+0.5)-((sqrt(3)-1)+0.5)/2)*sin(180)}) arc (180:{180+atan(2/(2/((sqrt(3)-1)+0.5)-((sqrt(3)-1)+0.5)/2))}:{2/((sqrt(3)-1)+0.5)-((sqrt(3)-1)+0.5)/2});
        \draw ({-((sqrt(3)-1)+2/((sqrt(3)-1)+0.5)-((sqrt(3)-1)+0.5)/2)+(2/((sqrt(3)-1)+0.5)-((sqrt(3)-1)+0.5)/2)*cos(0)},{2+(2/((sqrt(3)-1)+0.5)-((sqrt(3)-1)+0.5)/2)*sin(0)}) arc (0:{-atan(2/(2/((sqrt(3)-1)+0.5)-((sqrt(3)-1)+0.5)/2))}:{2/((sqrt(3)-1)+0.5)-((sqrt(3)-1)+0.5)/2});

        \draw ({sqrt(3)-1},-2) arc (0:-180:{sqrt(3)-1} and 0.2);
        \draw[dashed] ({sqrt(3)-1},-2) arc (0:180:{sqrt(3)-1} and 0.2);
        \draw ({0.5+2*(sqrt(3)-1)},0) arc (0:-180:{sqrt(3)-1} and 0.2);
        \draw[dashed] ({0.5+2*(sqrt(3)-1)},0) arc (0:180:{sqrt(3)-1} and 0.2);
        \draw (-0.5,0) arc (0:-180:{sqrt(3)-1} and 0.2);
        \draw[dashed] (-0.5,0) arc (0:180:{sqrt(3)-1} and 0.2);
        \draw (0,2) ellipse ({sqrt(3)-1} and 0.2);
    \end{tikzpicture}
    \ar[d,dotted,leftrightarrow]
    \\
    \begin{tikzpicture}[thick]            
        \draw ({2*sqrt(3)/2+cos(240)},{1+sin(240)}) arc (240:300:1);
        \draw ({cos(-20)},{-2+sin(-20)}) arc (-20:60:1);
        \draw ({2*sqrt(3)+cos(120)},{-2+sin(120)}) arc (120:{180+20}:1);

        \draw[rotate around={240:({cos(330)},{sin(330)})}] ({cos(330)+cos(0)*(sqrt(3)-1)},{sin(330)+sin(0)*(sqrt(3)-1)}) arc (0:180:{sqrt(3)-1} and 0.2);
        \draw[rotate around={240:({cos(330)},{sin(330)})}] ({cos(330)+cos(180)*(sqrt(3)-1)},{sin(330)+sin(180)*(sqrt(3)-1)}) arc (180:360:{sqrt(3)-1} and 0.2);
        
        \draw[rotate around={-60:({sqrt(3)+cos(30)},{-1+sin(30)})}] ({sqrt(3)+cos(30)},{-1+sin(30)}) ellipse ({sqrt(3)-1} and 0.2);
        
        \draw[dashed] ({sqrt(3)+cos(270)+cos(0)*(sqrt(3)-1)},{-1+sin(270)+sin(0)*(sqrt(3)-1)}) arc (0:180:{sqrt(3)-1} and 0.2);
        \draw ({sqrt(3)+cos(270)+cos(180)*(sqrt(3)-1)},{-1+sin(270)+sin(180)*(sqrt(3)-1)}) arc (180:360:{sqrt(3)-1} and 0.2);

        \draw[rotate around={60:({cos(330)},{sin(330)})}] ({cos(330)+cos(0)*(sqrt(3)-1)+0.5},{sin(330)+sin(0)*(sqrt(3)-1)}) circle (0.5);
    \end{tikzpicture}
    \ar[r,dotted,leftrightarrow]
    &
    \begin{tikzpicture}[thick]
        \draw (0,0) circle (0.5);
        \draw ({sqrt(3)-1+(0.5+(sqrt(3)-1))*cos(-atan(2/(2/((sqrt(3)-1)+0.5)-((sqrt(3)-1)+0.5)/2)))},{(0.5+(sqrt(3)-1))*sin(-atan(2/(2/((sqrt(3)-1)+0.5)-((sqrt(3)-1)+0.5)/2)))}) arc ({-atan(2/(2/((sqrt(3)-1)+0.5)-((sqrt(3)-1)+0.5)/2))}:0:{0.5+(sqrt(3)-1)});
        \draw ({-(sqrt(3)-1)+(0.5+(sqrt(3)-1))*cos(180)},{(0.5+(sqrt(3)-1))*sin(180)}) arc (180:{180+atan(2/(2/((sqrt(3)-1)+0.5)-((sqrt(3)-1)+0.5)/2))}:{0.5+(sqrt(3)-1)});

        \draw ({(sqrt(3)-1)+2/((sqrt(3)-1)+0.5)-((sqrt(3)-1)+0.5)/2+(2/((sqrt(3)-1)+0.5)-((sqrt(3)-1)+0.5)/2)*cos(180-atan(2/(2/((sqrt(3)-1)+0.5)-((sqrt(3)-1)+0.5)/2)))},{-2+(2/((sqrt(3)-1)+0.5)-((sqrt(3)-1)+0.5)/2)*sin(180-atan(2/(2/((sqrt(3)-1)+0.5)-((sqrt(3)-1)+0.5)/2)))}) arc ({180-atan(2/(2/((sqrt(3)-1)+0.5)-((sqrt(3)-1)+0.5)/2))}:200:{2/((sqrt(3)-1)+0.5)-((sqrt(3)-1)+0.5)/2});
        \draw ({-((sqrt(3)-1)+2/((sqrt(3)-1)+0.5)-((sqrt(3)-1)+0.5)/2)+(2/((sqrt(3)-1)+0.5)-((sqrt(3)-1)+0.5)/2)*cos(-20)},{-2+(2/((sqrt(3)-1)+0.5)-((sqrt(3)-1)+0.5)/2)*sin(-20)}) arc (-20:{atan(2/(2/((sqrt(3)-1)+0.5)-((sqrt(3)-1)+0.5)/2))}:{2/((sqrt(3)-1)+0.5)-((sqrt(3)-1)+0.5)/2});

        \draw ({sqrt(3)-1},-2) arc (0:-180:{sqrt(3)-1} and 0.2);
        \draw[dashed] ({sqrt(3)-1},-2) arc (0:180:{sqrt(3)-1} and 0.2);
        \draw ({0.5+2*(sqrt(3)-1)},0) arc (0:-180:{sqrt(3)-1} and 0.2);
        \draw ({0.5+2*(sqrt(3)-1)},0) arc (0:180:{sqrt(3)-1} and 0.2);
        \draw (-0.5,0) arc (0:-180:{sqrt(3)-1} and 0.2);
        \draw (-0.5,0) arc (0:180:{sqrt(3)-1} and 0.2);
    \end{tikzpicture}
\end{tikzcd}
    \caption{Elementary move II}
    \label{fig:tying}
\end{figure}
Then there exists the unique boundary curve $\beta \subset X$ such that the union of $\eta(P_X)$ and $\eta(\beta)$ is a pants decomposition of $Y$, denoted by $P_Y$.
We say \emph{Elementary move II} is the move that modifies $[P_X, h_X, X]$ to $[P_Y, \eta \circ h_X, Y]$.
We also call the inverse of this move \emph{Elementary move II}.

\subsubsection{Elementary move III}

Elementary move III is one of the elementary moves introduced by Papadopoulos--Penner \cite{MR3460762}.
This move exchanges two one-sided curves in a pair of pants with a single two-sided curve, as shown in Figure \ref{fig:moveiii}.
Importantly, this move does not modify the underlying surface or affect the homotopy equivalence.

\begin{figure}[ht]
    \centering
    \begin{tikzcd}
    \begin{tikzpicture}[thick]
        \draw ({cos(240)},{2+sin(240)}) arc (240:300:1);
        \draw ({-sqrt(3)+cos(-20)},{-1+sin(-20)}) arc (-20:60:1);
        \draw ({sqrt(3)+cos(120)},{-1+sin(120)}) arc (120:{180+20}:1);

        \draw[ultra thick, rotate around={240:({cos(150)},{sin(150)})}] ({cos(150)},{sin(150)}) ellipse ({sqrt(3)-1} and 0.2);
        \draw[rotate around={240:({cos(150)},{sin(150)})}] ({cos(150)+(sqrt(3)-1)*cos(-135)},{sin(150)+0.2*sin(-135)})--({cos(150)+(sqrt(3)-1)*cos(45)},{sin(150)+0.2*sin(45)});
        \draw[rotate around={240:({cos(150)},{sin(150)})}] ({cos(150)+(sqrt(3)-1)*cos(-45)},{sin(150)+0.2*sin(-45)})--({cos(150)+(sqrt(3)-1)*cos(135)},{sin(150)+0.2*sin(135)});
        
        \draw[ultra thick, rotate around={-60:({cos(30)},{sin(30)})}] ({cos(30)},{sin(30)}) ellipse ({sqrt(3)-1} and 0.2);
        \draw[rotate around={-60:({cos(30)},{sin(30)})}] ({cos(30)+(sqrt(3)-1)*cos(-135)},{sin(30)+0.2*sin(-135)})--({cos(30)+(sqrt(3)-1)*cos(45)},{sin(30)+0.2*sin(45)});
        \draw[rotate around={-60:({cos(30)},{sin(30)})}] ({cos(30)+(sqrt(3)-1)*cos(-45)},{sin(30)+0.2*sin(-45)})--({cos(30)+(sqrt(3)-1)*cos(135)},{sin(30)+0.2*sin(135)});
        
        \draw[dashed] ({cos(270)+cos(0)*(sqrt(3)-1)},{sin(270)+sin(0)*(sqrt(3)-1)}) arc (0:180:{sqrt(3)-1} and 0.2);
        \draw ({cos(270)+cos(180)*(sqrt(3)-1)},{sin(270)+sin(180)*(sqrt(3)-1)}) arc (180:360:{sqrt(3)-1} and 0.2);
    \end{tikzpicture}
    \ar[r, leftrightarrow, "\mathrm{III}"]
    &
    \begin{tikzpicture}[thick]
        \draw ({cos(240)},{2+sin(240)}) arc (240:300:1);
        \draw ({-sqrt(3)+cos(-20)},{-1+sin(-20)}) arc (-20:60:1);
        \draw ({sqrt(3)+cos(120)},{-1+sin(120)}) arc (120:{180+20}:1);

        \draw[rotate around={240:({cos(150)},{sin(150)})}] ({cos(150)},{sin(150)}) ellipse ({sqrt(3)-1} and 0.2);
        \draw[rotate around={240:({cos(150)},{sin(150)})}] ({cos(150)+(sqrt(3)-1)*cos(-135)},{sin(150)+0.2*sin(-135)})--({cos(150)+(sqrt(3)-1)*cos(45)},{sin(150)+0.2*sin(45)});
        \draw[rotate around={240:({cos(150)},{sin(150)})}] ({cos(150)+(sqrt(3)-1)*cos(-45)},{sin(150)+0.2*sin(-45)})--({cos(150)+(sqrt(3)-1)*cos(135)},{sin(150)+0.2*sin(135)});
        
        \draw[rotate around={-60:({cos(30)},{sin(30)})}] ({cos(30)},{sin(30)}) ellipse ({sqrt(3)-1} and 0.2);
        \draw[rotate around={-60:({cos(30)},{sin(30)})}] ({cos(30)+(sqrt(3)-1)*cos(-135)},{sin(30)+0.2*sin(-135)})--({cos(30)+(sqrt(3)-1)*cos(45)},{sin(30)+0.2*sin(45)});
        \draw[rotate around={-60:({cos(30)},{sin(30)})}] ({cos(30)+(sqrt(3)-1)*cos(-45)},{sin(30)+0.2*sin(-45)})--({cos(30)+(sqrt(3)-1)*cos(135)},{sin(30)+0.2*sin(135)});
        
        \draw[dashed] ({cos(270)+cos(0)*(sqrt(3)-1)},{sin(270)+sin(0)*(sqrt(3)-1)}) arc (0:180:{sqrt(3)-1} and 0.2);
        \draw ({cos(270)+cos(180)*(sqrt(3)-1)},{sin(270)+sin(180)*(sqrt(3)-1)}) arc (180:360:{sqrt(3)-1} and 0.2);

        \draw ({cos(150)+(sqrt(3)-1)*cos(90)*cos(240)-0.2*sin(90)*sin(240)},{sin(150)+(sqrt(3)-1)*cos(90)*sin(240)+0.2*sin(90)*cos(240)}) to [bend right=20] ({cos(30)+(sqrt(3)-1)*cos(-90)*cos(-60)-0.2*sin(-90)*sin(-60)},{sin(30)+(sqrt(3)-1)*cos(-90)*sin(-60)+0.2*sin(-90)*cos(-60)});
        \draw[dashed] ({cos(150)+(sqrt(3)-1)*cos(-90)*cos(240)-0.2*sin(-90)*sin(240)},{sin(150)+(sqrt(3)-1)*cos(-90)*sin(240)+0.2*sin(-90)*cos(240)}) to [bend right=20] ({cos(30)+(sqrt(3)-1)*cos(90)*cos(-60)-0.2*sin(90)*sin(-60)},{sin(30)+(sqrt(3)-1)*cos(90)*sin(-60)+0.2*sin(90)*cos(-60)});
    \end{tikzpicture}
\end{tikzcd}
    \caption{Elementary move III}
    \label{fig:moveiii}
\end{figure}

\subsubsection{Elementary move IV}

``Elementary move IV'' deforms a pair of pants into a twice-holed real projective plane.
Consider a pants decomposition $[P_X, h_X, X]$ in which there are two boundary curves, $\beta_0$ and $\beta_1$, and a simple closed curve $\gamma \subset P_X$ such that $\beta_0 \cup \beta_1 \cup \gamma$ forms a pair of pants.
As illustrated by Figure \ref{fig:moveiv}, Elementary move IV produces another pants decomposition $[P_Y, h_Y, Y]$ from $[P_X, h_X, X]$ through the following process:
\begin{enumerate}
    \item First, deform the pair of pants into a figure-eight shape using a deformation retract.
    \item Then, embed this figure-eight into a surface $Y$ via another deformation retract, where a regular neighborhood of the figure-eight is homeomorphic to a twice-holed real projective plane.
    \item Let $h_Y$ denote the homotopy equivalence induced by the deformation retract, and let $P_Y$ be the image of $P_X$ under this transformation.
\end{enumerate}

\begin{figure}[ht]
    \centering
    \begin{tikzcd}
    \begin{tikzpicture}[thick]
        \draw ({cos(240)},{2+sin(240)}) arc (240:300:1);
        \draw ({-sqrt(3)+cos(-20)},{-1+sin(-20)}) arc (-20:60:1);
        \draw ({sqrt(3)+cos(120)},{-1+sin(120)}) arc (120:{180+20}:1);

        \draw[rotate around={240:({cos(150)},{sin(150)})}] ({cos(150)},{sin(150)}) ellipse ({sqrt(3)-1} and 0.2);            
        \draw[rotate around={-60:({cos(30)},{sin(30)})}] ({cos(30)},{sin(30)}) ellipse ({sqrt(3)-1} and 0.2);
        
        \draw[dashed] ({cos(270)+cos(0)*(sqrt(3)-1)},{sin(270)+sin(0)*(sqrt(3)-1)}) arc (0:180:{sqrt(3)-1} and 0.2);
        \draw ({cos(270)+cos(180)*(sqrt(3)-1)},{sin(270)+sin(180)*(sqrt(3)-1)}) arc (180:360:{sqrt(3)-1} and 0.2);
    \end{tikzpicture}
    \ar[rr, leftrightarrow, "\mathrm{IV}"]
    &&
    \begin{tikzpicture}[thick]
        \draw ({cos(240)},{2+sin(240)}) arc (240:300:1);
        \draw ({-sqrt(3)+cos(-20)},{-1+sin(-20)}) arc (-20:60:1);
        \draw ({sqrt(3)+cos(120)},{-1+sin(120)}) arc (120:{180+20}:1);

        \draw[ultra thick, rotate around={240:({cos(150)},{sin(150)})}] ({cos(150)},{sin(150)}) ellipse ({sqrt(3)-1} and 0.2);
        \draw[rotate around={240:({cos(150)},{sin(150)})}] ({cos(150)+(sqrt(3)-1)*cos(-135)},{sin(150)+0.2*sin(-135)})--({cos(150)+(sqrt(3)-1)*cos(45)},{sin(150)+0.2*sin(45)});
        \draw[rotate around={240:({cos(150)},{sin(150)})}] ({cos(150)+(sqrt(3)-1)*cos(-45)},{sin(150)+0.2*sin(-45)})--({cos(150)+(sqrt(3)-1)*cos(135)},{sin(150)+0.2*sin(135)});
        
        \draw[rotate around={-60:({cos(30)},{sin(30)})}] ({cos(30)},{sin(30)}) ellipse ({sqrt(3)-1} and 0.2);
        
        \draw[dashed] ({cos(270)+cos(0)*(sqrt(3)-1)},{sin(270)+sin(0)*(sqrt(3)-1)}) arc (0:180:{sqrt(3)-1} and 0.2);
        \draw ({cos(270)+cos(180)*(sqrt(3)-1)},{sin(270)+sin(180)*(sqrt(3)-1)}) arc (180:360:{sqrt(3)-1} and 0.2);
    \end{tikzpicture}
    \\&
    \begin{tikzpicture}[thick]
        \draw ({-sqrt(3)+cos(-20)},{-1+sin(-20)}) arc (-20:60:1);
        \draw ({sqrt(3)+cos(180)},{-1+sin(180)}) arc (180:{180+20}:1);
        \draw[rotate around={240:({cos(150)},{sin(150)})}] ({cos(150)},{sin(150)}) ellipse ({sqrt(3)-1} and 0.2);
        \draw ({cos(270)},{sin(270)}) ellipse ({sqrt(3)-1} and 0.2);
    \end{tikzpicture}
    \ar[ul, dashed, hookrightarrow]
    \ar[ur, dashed, hookrightarrow]
    &
\end{tikzcd}
    \caption{Elementary move IV}
    \label{fig:moveiv}
\end{figure}

\subsection{(Non-orientable) Pants graphs}

We define a simplicial graph $\Pants_n$ as follows:
Consider all pants decompositions of the $n$-rose as the vertices of $\Pants_n$. 
Two vertices are joined by an edge in $\Pants_n$ if and only if there exist representatives of these vertices such that one is obtained from the other by Elementary move I, II, III, or IV.
We call $\Pants_n$ the \emph{pants graph} of $F_n$.

Next, we define a simplicial action of $\Out(F_n)$ on $\Pants_n$:
For each $\phi \in \Out(F_n)$ and $P \in \Pants_n$, let $f: R_n \to R_n$ be a representative of $\phi$, and let $(P_X, h_X, X)$ be a representative of $P$. 
We define $\phi.P$ as the equivalence class of the new pants decomposition $(P_X, h_X \circ f^{-1}, X)$.

\subsection{Orientable pants graph}

We also define a simplicial graph $\Pants^+_n$ for all pants decompositions of orientable surfaces marked by $R_n$, which forms an induced subgraph of $\Pants_n$.
In this case, the vertices of $\Pants^+_n$ are joined by edges using only Elementary moves I and II.
We call $\Pants_n^+$ the \emph{orientable pants graph} of $F_n$.

This construction is $\Out(F_n)$--invariant, so that one may regard it as an $\Out(F_n)$--graph. 
More precisely, for each $P \in \Pants_n^+$ and $\phi \in \Out(F_n)$, there exist representatives $(P_X,h_X,X) \in P$ and $f \in \phi$ such that $\phi.P$ has a representative of the form $(P_X, h_X \circ f^{-1}, X)$, which is again a pants decomposition of an orientable surface. 
Hence, $f.P \in \Pants_n^+$.

\begin{lemma} \label{lem:distance_three}
    Let \(h_X: R_n \to X\) be a homotopy equivalence into an orientable surface \(X\). Let \((P_X, h_X, X)\) and \((Q_X, h_X, X)\) be pants decompositions that differ by two curves, as illustrated in Figure \ref{fig:distance_three}. More precisely, assume the following.
    \begin{enumerate}
        \item There is a separating curve of \(P_X \cap Q_X\) that bounds a twice-holed torus \(Y \subset X\);
        \item \( P_X - Y = Q_X - Y \); and
        \item For all essential simple closed curves \(\alpha \subset P_X \cap Y\) and \(\beta \subset Q_X \cap Y\), their intersection number satisfies \(i(\alpha, \beta) = 1\).
    \end{enumerate}
    Then the distance in $\Pants_n^+$ between the vertices corresponding to $[P_X, h_X, X]$ and $[Q_X, h_X, X]$ is at most three.
    \begin{figure}[ht]
        \centering
        \begin{tikzcd}
    \begin{tikzpicture}[thick]
        \draw (0,0) circle (0.5);
        \draw ({sqrt(3)-1+(0.5+(sqrt(3)-1))*cos(-atan(2/(2/((sqrt(3)-1)+0.5)-((sqrt(3)-1)+0.5)/2)))},{(0.5+(sqrt(3)-1))*sin(-atan(2/(2/((sqrt(3)-1)+0.5)-((sqrt(3)-1)+0.5)/2)))}) arc ({-atan(2/(2/((sqrt(3)-1)+0.5)-((sqrt(3)-1)+0.5)/2))}:{atan(2/(2/((sqrt(3)-1)+0.5)-((sqrt(3)-1)+0.5)/2))}:{0.5+(sqrt(3)-1)});
        \draw ({-(sqrt(3)-1)+(0.5+(sqrt(3)-1))*cos(180-atan(2/(2/((sqrt(3)-1)+0.5)-((sqrt(3)-1)+0.5)/2)))},{(0.5+(sqrt(3)-1))*sin(180-atan(2/(2/((sqrt(3)-1)+0.5)-((sqrt(3)-1)+0.5)/2)))}) arc ({180-atan(2/(2/((sqrt(3)-1)+0.5)-((sqrt(3)-1)+0.5)/2))}:{180+atan(2/(2/((sqrt(3)-1)+0.5)-((sqrt(3)-1)+0.5)/2))}:{0.5+(sqrt(3)-1)});

        \draw ({(sqrt(3)-1)+2/((sqrt(3)-1)+0.5)-((sqrt(3)-1)+0.5)/2+(2/((sqrt(3)-1)+0.5)-((sqrt(3)-1)+0.5)/2)*cos(180-atan(2/(2/((sqrt(3)-1)+0.5)-((sqrt(3)-1)+0.5)/2)))},{-2+(2/((sqrt(3)-1)+0.5)-((sqrt(3)-1)+0.5)/2)*sin(180-atan(2/(2/((sqrt(3)-1)+0.5)-((sqrt(3)-1)+0.5)/2)))}) arc ({180-atan(2/(2/((sqrt(3)-1)+0.5)-((sqrt(3)-1)+0.5)/2))}:200:{2/((sqrt(3)-1)+0.5)-((sqrt(3)-1)+0.5)/2});
        \draw ({-((sqrt(3)-1)+2/((sqrt(3)-1)+0.5)-((sqrt(3)-1)+0.5)/2)+(2/((sqrt(3)-1)+0.5)-((sqrt(3)-1)+0.5)/2)*cos(-20)},{-2+(2/((sqrt(3)-1)+0.5)-((sqrt(3)-1)+0.5)/2)*sin(-20)}) arc (-20:{atan(2/(2/((sqrt(3)-1)+0.5)-((sqrt(3)-1)+0.5)/2))}:{2/((sqrt(3)-1)+0.5)-((sqrt(3)-1)+0.5)/2});
        \draw ({(sqrt(3)-1)+2/((sqrt(3)-1)+0.5)-((sqrt(3)-1)+0.5)/2+(2/((sqrt(3)-1)+0.5)-((sqrt(3)-1)+0.5)/2)*cos(180)},{2+(2/((sqrt(3)-1)+0.5)-((sqrt(3)-1)+0.5)/2)*sin(180)}) arc (180:{180+atan(2/(2/((sqrt(3)-1)+0.5)-((sqrt(3)-1)+0.5)/2))}:{2/((sqrt(3)-1)+0.5)-((sqrt(3)-1)+0.5)/2});
        \draw ({-((sqrt(3)-1)+2/((sqrt(3)-1)+0.5)-((sqrt(3)-1)+0.5)/2)+(2/((sqrt(3)-1)+0.5)-((sqrt(3)-1)+0.5)/2)*cos(0)},{2+(2/((sqrt(3)-1)+0.5)-((sqrt(3)-1)+0.5)/2)*sin(0)}) arc (0:{-atan(2/(2/((sqrt(3)-1)+0.5)-((sqrt(3)-1)+0.5)/2))}:{2/((sqrt(3)-1)+0.5)-((sqrt(3)-1)+0.5)/2});

        \draw ({sqrt(3)-1},-2) arc (0:-180:{sqrt(3)-1} and 0.2);
        \draw[dashed] ({sqrt(3)-1},-2) arc (0:180:{sqrt(3)-1} and 0.2);
        \draw ({0.5+2*(sqrt(3)-1)},0) arc (0:-180:{sqrt(3)-1} and 0.2);
        \draw[dashed] ({0.5+2*(sqrt(3)-1)},0) arc (0:180:{sqrt(3)-1} and 0.2);
        \draw (-0.5,0) arc (0:-180:{sqrt(3)-1} and 0.2);
        \draw[dashed] (-0.5,0) arc (0:180:{sqrt(3)-1} and 0.2);
        \draw (0,2) ellipse ({sqrt(3)-1} and 0.2);
    \end{tikzpicture}
    \ar[r, dashed, leftrightarrow]
    &
    \begin{tikzpicture}[thick]
        \draw (0,0) circle (0.5);
        \draw ({sqrt(3)-1+(0.5+(sqrt(3)-1))*cos(-atan(2/(2/((sqrt(3)-1)+0.5)-((sqrt(3)-1)+0.5)/2)))},{(0.5+(sqrt(3)-1))*sin(-atan(2/(2/((sqrt(3)-1)+0.5)-((sqrt(3)-1)+0.5)/2)))}) arc ({-atan(2/(2/((sqrt(3)-1)+0.5)-((sqrt(3)-1)+0.5)/2))}:{atan(2/(2/((sqrt(3)-1)+0.5)-((sqrt(3)-1)+0.5)/2))}:{0.5+(sqrt(3)-1)});
        \draw ({-(sqrt(3)-1)+(0.5+(sqrt(3)-1))*cos(180-atan(2/(2/((sqrt(3)-1)+0.5)-((sqrt(3)-1)+0.5)/2)))},{(0.5+(sqrt(3)-1))*sin(180-atan(2/(2/((sqrt(3)-1)+0.5)-((sqrt(3)-1)+0.5)/2)))}) arc ({180-atan(2/(2/((sqrt(3)-1)+0.5)-((sqrt(3)-1)+0.5)/2))}:{180+atan(2/(2/((sqrt(3)-1)+0.5)-((sqrt(3)-1)+0.5)/2))}:{0.5+(sqrt(3)-1)});

        \draw ({(sqrt(3)-1)+2/((sqrt(3)-1)+0.5)-((sqrt(3)-1)+0.5)/2+(2/((sqrt(3)-1)+0.5)-((sqrt(3)-1)+0.5)/2)*cos(180-atan(2/(2/((sqrt(3)-1)+0.5)-((sqrt(3)-1)+0.5)/2)))},{-2+(2/((sqrt(3)-1)+0.5)-((sqrt(3)-1)+0.5)/2)*sin(180-atan(2/(2/((sqrt(3)-1)+0.5)-((sqrt(3)-1)+0.5)/2)))}) arc ({180-atan(2/(2/((sqrt(3)-1)+0.5)-((sqrt(3)-1)+0.5)/2))}:200:{2/((sqrt(3)-1)+0.5)-((sqrt(3)-1)+0.5)/2});
        \draw ({-((sqrt(3)-1)+2/((sqrt(3)-1)+0.5)-((sqrt(3)-1)+0.5)/2)+(2/((sqrt(3)-1)+0.5)-((sqrt(3)-1)+0.5)/2)*cos(-20)},{-2+(2/((sqrt(3)-1)+0.5)-((sqrt(3)-1)+0.5)/2)*sin(-20)}) arc (-20:{atan(2/(2/((sqrt(3)-1)+0.5)-((sqrt(3)-1)+0.5)/2))}:{2/((sqrt(3)-1)+0.5)-((sqrt(3)-1)+0.5)/2});
        \draw ({(sqrt(3)-1)+2/((sqrt(3)-1)+0.5)-((sqrt(3)-1)+0.5)/2+(2/((sqrt(3)-1)+0.5)-((sqrt(3)-1)+0.5)/2)*cos(180)},{2+(2/((sqrt(3)-1)+0.5)-((sqrt(3)-1)+0.5)/2)*sin(180)}) arc (180:{180+atan(2/(2/((sqrt(3)-1)+0.5)-((sqrt(3)-1)+0.5)/2))}:{2/((sqrt(3)-1)+0.5)-((sqrt(3)-1)+0.5)/2});
        \draw ({-((sqrt(3)-1)+2/((sqrt(3)-1)+0.5)-((sqrt(3)-1)+0.5)/2)+(2/((sqrt(3)-1)+0.5)-((sqrt(3)-1)+0.5)/2)*cos(0)},{2+(2/((sqrt(3)-1)+0.5)-((sqrt(3)-1)+0.5)/2)*sin(0)}) arc (0:{-atan(2/(2/((sqrt(3)-1)+0.5)-((sqrt(3)-1)+0.5)/2))}:{2/((sqrt(3)-1)+0.5)-((sqrt(3)-1)+0.5)/2});

        \draw ({sqrt(3)-1},-2) arc (0:-180:{sqrt(3)-1} and 0.2);
        \draw[dashed] ({sqrt(3)-1},-2) arc (0:180:{sqrt(3)-1} and 0.2);
        \draw (0,0) circle (1);
        \draw (-1.4,0) arc (-180:0:1.4) -- (1.4,0.8) arc (0:35:1.4 and 0.673);
        \draw (-1.4,0) -- (-1.4,0.8) arc (180:145:1.4 and 0.673);
        \draw[dashed] (1.4,0.8) arc (0:180:1.4 and 0.673);
        \draw (0,2) ellipse ({sqrt(3)-1} and 0.2);
    \end{tikzpicture}
\end{tikzcd}
        \caption{Their distance in $\Pants_n^+$ is at most three.}
        \label{fig:distance_three}
    \end{figure}
\end{lemma}

\begin{proof}
    \begin{figure}[ht]
        \centering
        \begin{tikzcd}
    \begin{tikzpicture}[thick]
        \draw (0,0) circle (0.5);
        \draw ({sqrt(3)-1+(0.5+(sqrt(3)-1))*cos(-atan(2/(2/((sqrt(3)-1)+0.5)-((sqrt(3)-1)+0.5)/2)))},{(0.5+(sqrt(3)-1))*sin(-atan(2/(2/((sqrt(3)-1)+0.5)-((sqrt(3)-1)+0.5)/2)))}) arc ({-atan(2/(2/((sqrt(3)-1)+0.5)-((sqrt(3)-1)+0.5)/2))}:{atan(2/(2/((sqrt(3)-1)+0.5)-((sqrt(3)-1)+0.5)/2))}:{0.5+(sqrt(3)-1)});
        \draw ({-(sqrt(3)-1)+(0.5+(sqrt(3)-1))*cos(180-atan(2/(2/((sqrt(3)-1)+0.5)-((sqrt(3)-1)+0.5)/2)))},{(0.5+(sqrt(3)-1))*sin(180-atan(2/(2/((sqrt(3)-1)+0.5)-((sqrt(3)-1)+0.5)/2)))}) arc ({180-atan(2/(2/((sqrt(3)-1)+0.5)-((sqrt(3)-1)+0.5)/2))}:{180+atan(2/(2/((sqrt(3)-1)+0.5)-((sqrt(3)-1)+0.5)/2))}:{0.5+(sqrt(3)-1)});
    
        \draw ({(sqrt(3)-1)+2/((sqrt(3)-1)+0.5)-((sqrt(3)-1)+0.5)/2+(2/((sqrt(3)-1)+0.5)-((sqrt(3)-1)+0.5)/2)*cos(180-atan(2/(2/((sqrt(3)-1)+0.5)-((sqrt(3)-1)+0.5)/2)))},{-2+(2/((sqrt(3)-1)+0.5)-((sqrt(3)-1)+0.5)/2)*sin(180-atan(2/(2/((sqrt(3)-1)+0.5)-((sqrt(3)-1)+0.5)/2)))}) arc ({180-atan(2/(2/((sqrt(3)-1)+0.5)-((sqrt(3)-1)+0.5)/2))}:200:{2/((sqrt(3)-1)+0.5)-((sqrt(3)-1)+0.5)/2});
        \draw ({-((sqrt(3)-1)+2/((sqrt(3)-1)+0.5)-((sqrt(3)-1)+0.5)/2)+(2/((sqrt(3)-1)+0.5)-((sqrt(3)-1)+0.5)/2)*cos(-20)},{-2+(2/((sqrt(3)-1)+0.5)-((sqrt(3)-1)+0.5)/2)*sin(-20)}) arc (-20:{atan(2/(2/((sqrt(3)-1)+0.5)-((sqrt(3)-1)+0.5)/2))}:{2/((sqrt(3)-1)+0.5)-((sqrt(3)-1)+0.5)/2});
        \draw ({(sqrt(3)-1)+2/((sqrt(3)-1)+0.5)-((sqrt(3)-1)+0.5)/2+(2/((sqrt(3)-1)+0.5)-((sqrt(3)-1)+0.5)/2)*cos(180)},{2+(2/((sqrt(3)-1)+0.5)-((sqrt(3)-1)+0.5)/2)*sin(180)}) arc (180:{180+atan(2/(2/((sqrt(3)-1)+0.5)-((sqrt(3)-1)+0.5)/2))}:{2/((sqrt(3)-1)+0.5)-((sqrt(3)-1)+0.5)/2});
        \draw ({-((sqrt(3)-1)+2/((sqrt(3)-1)+0.5)-((sqrt(3)-1)+0.5)/2)+(2/((sqrt(3)-1)+0.5)-((sqrt(3)-1)+0.5)/2)*cos(0)},{2+(2/((sqrt(3)-1)+0.5)-((sqrt(3)-1)+0.5)/2)*sin(0)}) arc (0:{-atan(2/(2/((sqrt(3)-1)+0.5)-((sqrt(3)-1)+0.5)/2))}:{2/((sqrt(3)-1)+0.5)-((sqrt(3)-1)+0.5)/2});
    
        \draw ({sqrt(3)-1},-2) arc (0:-180:{sqrt(3)-1} and 0.2);
        \draw[dashed] ({sqrt(3)-1},-2) arc (0:180:{sqrt(3)-1} and 0.2);
        \draw (-0.5,0) arc (0:-180:{sqrt(3)-1} and 0.2);
        \draw[dashed] (-0.5,0) arc (0:180:{sqrt(3)-1} and 0.2);
        \draw (0,2) ellipse ({sqrt(3)-1} and 0.2);
    
        \draw [ultra thick, red] ({0.5+2*(sqrt(3)-1)},0) arc (0:-180:{sqrt(3)-1} and 0.2);
        \draw[ultra thick, dashed, red] ({0.5+2*(sqrt(3)-1)},0) arc (0:180:{sqrt(3)-1} and 0.2);
        \draw[ultra thick, red] (0,0) circle (1);
        \draw[ultra thick, red] (-1.9, -0.4) to[out=-65, in={-150-90}] (-150:1.4);
        \draw[ultra thick, red] (-150:1.4) arc (-150:150:1.4);
        \draw[ultra thick, red] (150:1.4) to[out={150+90}, in={150-90}] (150:0.5);
        \draw[ultra thick, red, dashed] (150:0.5) to[out=240, in=115] (-1.9,-0.4);
    \end{tikzpicture}
    \ar[r, leftrightarrow, "\mathrm{II}"] &
    \begin{tikzpicture}[thick]
        \draw ({-2*sqrt(3)/2+cos(-60)},{1+sin(-60)}) arc (-60:0:1);
        \draw ({2*sqrt(3)/2+cos(180)},{1+sin(180)}) arc (180:300:1);
        \draw ({cos(-20)},{-2+sin(-20)}) arc (-20:120:1);
        \draw ({2*sqrt(3)+cos(120)},{-2+sin(120)}) arc (120:{180+20}:1);
        
        \draw[ultra thick, red] (0,1) ellipse ({sqrt(3)-1} and 0.2);
        \draw[ultra thick, red, rotate around={120:({cos(210)},{sin(210)})}] ({cos(210)},{sin(210)}) ellipse ({sqrt(3)-1} and 0.2);
        
        \draw[rotate around={240:({cos(330)},{sin(330)})}] ({cos(330)+cos(0)*(sqrt(3)-1)},{sin(330)+sin(0)*(sqrt(3)-1)}) arc (0:180:{sqrt(3)-1} and 0.2);
        \draw[dashed, rotate around={240:({cos(330)},{sin(330)})}] ({cos(330)+cos(180)*(sqrt(3)-1)},{sin(330)+sin(180)*(sqrt(3)-1)}) arc (180:360:{sqrt(3)-1} and 0.2);
        
        \draw[ultra thick, red, rotate around={-60:({sqrt(3)+cos(30)},{-1+sin(30)})}] ({sqrt(3)+cos(30)},{-1+sin(30)}) ellipse ({sqrt(3)-1} and 0.2);
        
        \draw[dashed] ({sqrt(3)+cos(270)+cos(0)*(sqrt(3)-1)},{-1+sin(270)+sin(0)*(sqrt(3)-1)}) arc (0:180:{sqrt(3)-1} and 0.2);
        \draw ({sqrt(3)+cos(270)+cos(180)*(sqrt(3)-1)},{-1+sin(270)+sin(180)*(sqrt(3)-1)}) arc (180:360:{sqrt(3)-1} and 0.2);
    \end{tikzpicture}
    \ar[d, leftrightarrow, "\mathrm{I}"] \\
    \begin{tikzpicture}[thick]
        \draw (0,0) circle (0.5);
        \draw ({sqrt(3)-1+(0.5+(sqrt(3)-1))*cos(-atan(2/(2/((sqrt(3)-1)+0.5)-((sqrt(3)-1)+0.5)/2)))},{(0.5+(sqrt(3)-1))*sin(-atan(2/(2/((sqrt(3)-1)+0.5)-((sqrt(3)-1)+0.5)/2)))}) arc ({-atan(2/(2/((sqrt(3)-1)+0.5)-((sqrt(3)-1)+0.5)/2))}:{atan(2/(2/((sqrt(3)-1)+0.5)-((sqrt(3)-1)+0.5)/2))}:{0.5+(sqrt(3)-1)});
        \draw ({-(sqrt(3)-1)+(0.5+(sqrt(3)-1))*cos(180-atan(2/(2/((sqrt(3)-1)+0.5)-((sqrt(3)-1)+0.5)/2)))},{(0.5+(sqrt(3)-1))*sin(180-atan(2/(2/((sqrt(3)-1)+0.5)-((sqrt(3)-1)+0.5)/2)))}) arc ({180-atan(2/(2/((sqrt(3)-1)+0.5)-((sqrt(3)-1)+0.5)/2))}:{180+atan(2/(2/((sqrt(3)-1)+0.5)-((sqrt(3)-1)+0.5)/2))}:{0.5+(sqrt(3)-1)});
    
        \draw ({(sqrt(3)-1)+2/((sqrt(3)-1)+0.5)-((sqrt(3)-1)+0.5)/2+(2/((sqrt(3)-1)+0.5)-((sqrt(3)-1)+0.5)/2)*cos(180-atan(2/(2/((sqrt(3)-1)+0.5)-((sqrt(3)-1)+0.5)/2)))},{-2+(2/((sqrt(3)-1)+0.5)-((sqrt(3)-1)+0.5)/2)*sin(180-atan(2/(2/((sqrt(3)-1)+0.5)-((sqrt(3)-1)+0.5)/2)))}) arc ({180-atan(2/(2/((sqrt(3)-1)+0.5)-((sqrt(3)-1)+0.5)/2))}:200:{2/((sqrt(3)-1)+0.5)-((sqrt(3)-1)+0.5)/2});
        \draw ({-((sqrt(3)-1)+2/((sqrt(3)-1)+0.5)-((sqrt(3)-1)+0.5)/2)+(2/((sqrt(3)-1)+0.5)-((sqrt(3)-1)+0.5)/2)*cos(-20)},{-2+(2/((sqrt(3)-1)+0.5)-((sqrt(3)-1)+0.5)/2)*sin(-20)}) arc (-20:{atan(2/(2/((sqrt(3)-1)+0.5)-((sqrt(3)-1)+0.5)/2))}:{2/((sqrt(3)-1)+0.5)-((sqrt(3)-1)+0.5)/2});
        \draw ({(sqrt(3)-1)+2/((sqrt(3)-1)+0.5)-((sqrt(3)-1)+0.5)/2+(2/((sqrt(3)-1)+0.5)-((sqrt(3)-1)+0.5)/2)*cos(180)},{2+(2/((sqrt(3)-1)+0.5)-((sqrt(3)-1)+0.5)/2)*sin(180)}) arc (180:{180+atan(2/(2/((sqrt(3)-1)+0.5)-((sqrt(3)-1)+0.5)/2))}:{2/((sqrt(3)-1)+0.5)-((sqrt(3)-1)+0.5)/2});
        \draw ({-((sqrt(3)-1)+2/((sqrt(3)-1)+0.5)-((sqrt(3)-1)+0.5)/2)+(2/((sqrt(3)-1)+0.5)-((sqrt(3)-1)+0.5)/2)*cos(0)},{2+(2/((sqrt(3)-1)+0.5)-((sqrt(3)-1)+0.5)/2)*sin(0)}) arc (0:{-atan(2/(2/((sqrt(3)-1)+0.5)-((sqrt(3)-1)+0.5)/2))}:{2/((sqrt(3)-1)+0.5)-((sqrt(3)-1)+0.5)/2});
    
        \draw ({sqrt(3)-1},-2) arc (0:-180:{sqrt(3)-1} and 0.2);
        \draw[dashed] ({sqrt(3)-1},-2) arc (0:180:{sqrt(3)-1} and 0.2);
        % \draw (-0.5,0) arc (0:-180:{sqrt(3)-1} and 0.2);
        % \draw[dashed] (-0.5,0) arc (0:180:{sqrt(3)-1} and 0.2);
        \draw (0,2) ellipse ({sqrt(3)-1} and 0.2);
    
        \draw [ultra thick, red] ({0.5+2*(sqrt(3)-1)},0) arc (0:-180:{sqrt(3)-1} and 0.2);
        \draw[ultra thick, dashed, red] ({0.5+2*(sqrt(3)-1)},0) arc (0:180:{sqrt(3)-1} and 0.2);
        \draw[ultra thick, red] (0,0) circle (1);
        \draw[ultra thick, red] (-1.9, -0.4) to[out=-65, in={-150-90}] (-150:1.4);
        \draw[ultra thick, red] (-150:1.4) arc (-150:150:1.4);
        \draw[ultra thick, red] (150:1.4) to[out={150+90}, in={150-90}] (150:0.5);
        \draw[ultra thick, red, dashed] (150:0.5) to[out=240, in=115] (-1.9,-0.4);

        \draw (-1.4,0.2) arc (-180:0:1.4) -- (1.4,0.8) arc (0:35:1.4 and 0.673);
        \draw (-1.4,0.2) -- (-1.4,0.8) arc (180:145:1.4 and 0.673);
        \draw[dashed] (1.4,0.8) arc (0:180:1.4 and 0.673);
    \end{tikzpicture}
    \ar[r, leftrightarrow, "\mathrm{II}"]
    &
    \begin{tikzpicture}[thick]
        \draw ({-2*sqrt(3)/2+cos(-60)},{1+sin(-60)}) arc (-60:0:1);
        \draw ({2*sqrt(3)/2+cos(180)},{1+sin(180)}) arc (180:300:1);
        \draw ({cos(-20)},{-2+sin(-20)}) arc (-20:120:1);
        \draw ({2*sqrt(3)+cos(120)},{-2+sin(120)}) arc (120:{180+20}:1);
        
        \draw[ultra thick, red] (0,1) ellipse ({sqrt(3)-1} and 0.2);
        \draw[ultra thick, red, rotate around={120:({cos(210)},{sin(210)})}] ({cos(210)},{sin(210)}) ellipse ({sqrt(3)-1} and 0.2);
        
        \draw[rotate around={150:({cos(330)},{sin(330)})}] ({cos(330)+cos(0)*2},{sin(330)+sin(0)*2}) arc (0:180:2 and 0.2);
        \draw[dashed, rotate around={150:({cos(330)},{sin(330)})}] ({cos(330)+cos(180)*2},{sin(330)+sin(180)*2}) arc (180:360:2 and 0.2);
        
        \draw[ultra thick, red, rotate around={-60:({sqrt(3)+cos(30)},{-1+sin(30)})}] ({sqrt(3)+cos(30)},{-1+sin(30)}) ellipse ({sqrt(3)-1} and 0.2);
        
        \draw[dashed] ({sqrt(3)+cos(270)+cos(0)*(sqrt(3)-1)},{-1+sin(270)+sin(0)*(sqrt(3)-1)}) arc (0:180:{sqrt(3)-1} and 0.2);
        \draw ({sqrt(3)+cos(270)+cos(180)*(sqrt(3)-1)},{-1+sin(270)+sin(180)*(sqrt(3)-1)}) arc (180:360:{sqrt(3)-1} and 0.2);
    \end{tikzpicture}
\end{tikzcd}
        \caption{Proof of Lemma \ref{lem:distance_three}}
        \label{fig:red_curves}
    \end{figure}
    To organize the argument, let $\alpha_1$ and $\alpha_2$ be the two simple closed curves of $P_X \cap Y$, and let $\beta_1$ and $\beta_2$ be the two simple closed curves of $Q_X \cap Y$.
    By the assumption of Lemma \ref{lem:distance_three}, there exists an essential simple closed curve $\gamma \subset Y$ such that
    \[
        i(\alpha_i,\gamma)=1=i(\beta_i,\gamma) \quad\text{for each } i=1,2.
    \]
    
    We first perform an Elementary move II on $(P_X,h_X,X)$ along the subsurface $Y$ so that the curves $\alpha_2$, $\beta_2$, and $\gamma$ become boundary components of the resulting marked surface.
    Denote the resulting pants decomposition by $(P_Z,h_Z,Z)$.
    Similarly, applying an Elementary move II to $(Q_X,h_X,X)$, we obtain another marked surface $(Q_{Z'},h_{Z'},Z')$ in which $\alpha_2$, $\beta_2$, and $\gamma$ are realized as boundary components.
    
    Since the three boundary curves $\alpha_2, \beta_2, \gamma$ determine the same marking data in both constructions, the marked surfaces $(h_Z, Z)$ and $(h_{Z'}, Z')$ are marking-equivalent.
    After identifying them, we may therefore regard both decompositions as lying on the same marked surface and write the second one simply as $(Q_Z,h_Z,Z)$.
    
    Since $i(\alpha_1,\beta_1)=1$, they have representatives contained in a common free basis of $F_n$.
    Thus, the images of $\alpha_1$ and $\beta_1$ in $Z$ intersect twice.
    So these two pants decompositions differ by an Elementary move I.
    Therefore,
    \[
    d_{\Pants_n^+}\bigl([P_X,h_X,X],[Q_X,h_X,X]\bigr)\le 3.
    \]
    This completes the proof.
\end{proof}

% \textcolor{red}{\textbf{Page 9, first paragraph of proof of Lemma 3.1:} I had to draw Figure 1 on my ipad to check this. It would be nice figures like this were just included whenever required.}

% \textcolor{red}{\textbf{Page 9, second paragraph of proof of Lemma 3.1:} ``Now locate the corresponding red curves on the right side of Figure 6.'' Please be more precise. I don’t know what this means.}

% \textcolor{blue}{I re-wrote the proof of Lemma 3.1.}

In Figure \ref{fig:distance_three}, one can get a common separating curve from both the left and the right configurations by applying an Elementary move I.
Once this curve is fixed, the remaining two curves intersect exactly once.
In particular, in such situations the interchange of two simple closed curves intersecting once determines a path in \(\Pants_n^+\) of length at most five.
In general, we can give a statement as follows.

\begin{proposition} \label{prop:surface_move}
    For each $n \geq 3$, there exists a constant $D = D(n) > 0$ with the following property.
    Let $h_X \colon R_n \to X$ be a homotopy equivalence to an orientable surface $X$.
    Suppose two pants decompositions $(P_X, h_X, X)$ and $(Q_X, h_X, X)$ differ in exactly one curve, with the differing pair having intersection number one as in Figure \ref{fig:surface_move_i}.
    More precisely, suppose the following.
    \begin{enumerate}
        \item There is a separating curve of $P_X \cap Q_X$ that bounds a once-holed torus $Y$;
        \item $P_X - Y = Q_X - Y$; and
        \item $i(P_X \cap Y, Q_X \cap Y) = 1$.
    \end{enumerate}
    Then the distance between the corresponding vertices in $\Pants_n^+$ is at most $D$.
    \begin{figure}[ht]
        \centering
        \begin{tikzcd}
    \begin{tikzpicture}[thick]
        \draw ({-sqrt(3)+cos(60)},{-1+sin(60)}) arc (60:-20:1);
        \draw ({sqrt(3)+cos(120)},{-1+sin(120)}) arc (120:{180+20}:1);
        \draw ({(2*sqrt(3)-1)*cos(-60)},{2+(2*sqrt(3)-1)*sin(-60)}) arc (-60:240:{2*sqrt(3)-1});

        \draw ({(1)*cos(0-30)},{2+(1)*sin(0-30)}) arc ({(0-30)}:{180+30}:{1});
        \draw ({(1+0.2)*cos(180+20)},{2+(1+0.2)*sin(180+20)+sqrt((1+0.2)^2-(1*cos(30))^2)-1*sin(30)}) arc ({180+20}:{360-20}:{1+0.2});
        
        \draw[dashed] ({cos(270)+cos(0)*(sqrt(3)-1)},{sin(270)+sin(0)*(sqrt(3)-1)}) arc (0:180:{sqrt(3)-1} and 0.2);
        \draw ({cos(270)+cos(180)*(sqrt(3)-1)},{sin(270)+sin(180)*(sqrt(3)-1)}) arc (180:360:{sqrt(3)-1} and 0.2);

        \draw ({0.2*cos(-90)}, {2+sqrt(3)+(sqrt(3)-1)*sin(-90)}) arc (-90:90:0.2 and {sqrt(3)-1});
        \draw[dashed] ({0.2*cos(90)}, {2+sqrt(3)+(sqrt(3)-1)*sin(90)}) arc (90:270:0.2 and {sqrt(3)-1});
    \end{tikzpicture}
    \ar[r, dashed, leftrightarrow]
    &
    \begin{tikzpicture}[thick]
        \draw ({-sqrt(3)+cos(60)},{-1+sin(60)}) arc (60:-20:1);
        \draw ({sqrt(3)+cos(120)},{-1+sin(120)}) arc (120:{180+20}:1);
        \draw ({(2*sqrt(3)-1)*cos(-60)},{2+(2*sqrt(3)-1)*sin(-60)}) arc (-60:240:{2*sqrt(3)-1});

        \draw ({(1)*cos(0-30)},{2+(1)*sin(0-30)}) arc ({(0-30)}:{180+30}:{1});
        \draw ({(1+0.2)*cos(180+20)},{2+(1+0.2)*sin(180+20)+sqrt((1+0.2)^2-(1*cos(30))^2)-1*sin(30)}) arc ({180+20}:{360-20}:{1+0.2});
        
        \draw[dashed] ({cos(270)+cos(0)*(sqrt(3)-1)},{sin(270)+sin(0)*(sqrt(3)-1)}) arc (0:180:{sqrt(3)-1} and 0.2);
        \draw ({cos(270)+cos(180)*(sqrt(3)-1)},{sin(270)+sin(180)*(sqrt(3)-1)}) arc (180:360:{sqrt(3)-1} and 0.2);

        \draw (0,2) circle ({sqrt(3)});
    \end{tikzpicture}
\end{tikzcd}
        \caption{Their distance is uniformly bounded.}
        \label{fig:surface_move_i}
    \end{figure}
\end{proposition}

% \textcolor{blue}{I refined the statement of Proposition 3.2.}

\begin{proof}
Let $\gamma$ be the separating curve on the left-hand side of Figure \ref{fig:surface_move_i}. 
Cutting $X$ along $\gamma$, we obtain two subsurfaces $Y$ and $Z$, where $Y$ is the once-holed surface illustrated in Figure \ref{fig:surface_move_i} and $Z$ is the complementary component. 
Let $P_Z$ denote the pants decomposition of $Z$ induced by $P_X$.  

Observe that the quotient of the pants graph $\Pants(Z)$ by the mapping class group $\MCG(Z)$ has finite diameter, denoted by $D_1$, and that every edge in $\Pants(Z)$ which is not a loop is realized by an Elementary move I. 
Hence there exists a sequence of pants decompositions
\[
    P_Z = P_Z^0,\, P_Z^1,\, \dots,\, P_Z^d = P_Z',
\]
with $d \leq D_1$, such that each transition is an Elementary move I and, in $P_Z'$, the curve $\gamma$ together with a boundary component of $Z$ bounds a pair of pants.  

Extending this sequence to $X$, we obtain a sequence
\[
    P_X = P_X^0,\, P_X^1,\, \dots,\, P_X^d = P_X'.
\]
Similarly, we may extend the sequence $P_Z^0, \dots, P_Z^d$ to a sequence of pants decompositions
\[
    Q_X = Q_X^0,\, Q_X^1,\, \dots,\, Q_X^d = Q_X'
\]
of $X$. By the preceding argument, $P_X'$ and $Q_X'$ are at distance at most five in $\Pants_n^+$.  
Therefore, we conclude that there exists a path joining $[P_X,h_X,X]$ and $[Q_X,h_X,X]$ of length at most $2D_1 + 5$.
\end{proof}

\begin{remark} \label{rem:surface_pants_graph}
    For each marked orientable surface \((h_X,X)\), one can define
    \[
        H_X:\Pants(X) \to \Pants_n^+,\quad P_X\mapsto [P_X,h_X,X],
    \]
    with sending edges to a geodesic segment in \(\Pants_n^+\) between the corresponding vertices. By Proposition~\ref{prop:surface_move}, there exist constants \(A\ge 1\) and \(B\ge 0\), depending only on \(n\), such that \(H_X\) is \((A,B)\)–coarsely Lipschitz over all marked orientable surface \((h_X,X)\).
    Unfortunately, it is currently unknown whether \(H_X\) is a quasi–isometric embedding.
\end{remark}

In a similar way, one can show that there exists a constant \(D' = D'(n) > 0\) such that if two pants decompositions \((P_X,h_X,X)\) and \((Q_X,h_X,X)\) differ by a single one-sided curve, as illustrated in Figure \ref{fig:one_sided_move}, then the corresponding vertices in \(\Pants_n\) are at distance at most \(D'\).
We leave the precise proof of this statement as an exercise for the reader. 
\emph{Hint:} it suffices to construct a path between these two vertices that contains two consecutive edges associated with an Elementary move IV.

\begin{figure}[ht]
    \centering
    \begin{tikzcd}
    \begin{tikzpicture}[thick]
        \draw ({cos(40)},{-2+sin(40)}) arc (40:140:1);
        \draw ({-sqrt(3)+cos(0)},{1+sin(0)}) arc (0:-80:1);
        \draw ({sqrt(3)+cos(180)},{1+sin(180)}) arc (180:260:1);

        \draw[ultra thick] (0,1) ellipse ({sqrt(3)-1} and 0.2);
        \draw ({0+(sqrt(3)-1)*cos(45)},{1+0.2*sin(45)}) -- ({0+(sqrt(3)-1)*cos(225)},{1+0.2*sin(225)});
        \draw ({0+(sqrt(3)-1)*cos(135)},{1+0.2*sin(135)}) -- ({0+(sqrt(3)-1)*cos(315)},{1+0.2*sin(315)});
        
        \draw[dashed, rotate around={60:(0,0)}] ({cos(270)+cos(0)*(sqrt(3)-1)},{sin(270)+sin(0)*(sqrt(3)-1)}) arc (0:180:{sqrt(3)-1} and 0.2);
        \draw[rotate around={60:(0,0)}] ({cos(270)+cos(180)*(sqrt(3)-1)},{sin(270)+sin(180)*(sqrt(3)-1)}) arc (180:360:{sqrt(3)-1} and 0.2);

        \draw[dashed, rotate around={-60:(0,0)}] ({cos(270)+cos(0)*(sqrt(3)-1)},{sin(270)+sin(0)*(sqrt(3)-1)}) arc (0:180:{sqrt(3)-1} and 0.2);
        \draw[rotate around={-60:(0,0)}] ({cos(270)+cos(180)*(sqrt(3)-1)},{sin(270)+sin(180)*(sqrt(3)-1)}) arc (180:360:{sqrt(3)-1} and 0.2);
    \end{tikzpicture}
    \ar[r, dashed, leftrightarrow]
    &
    \begin{tikzpicture}[thick]
        \draw ({cos(40)},{-2+sin(40)}) arc (40:140:1);
        \draw ({-sqrt(3)+cos(0)},{1+sin(0)}) arc (0:-80:1);
        \draw ({sqrt(3)+cos(180)},{1+sin(180)}) arc (180:260:1);

        \draw (0,1) ellipse ({sqrt(3)-1} and 0.2);
        \draw ({0+(sqrt(3)-1)*cos(45)},{1+0.2*sin(45)}) -- ({0+(sqrt(3)-1)*cos(225)},{1+0.2*sin(225)});
        \draw ({0+(sqrt(3)-1)*cos(135)},{1+0.2*sin(135)}) -- ({0+(sqrt(3)-1)*cos(315)},{1+0.2*sin(315)});

        \draw ({0+(sqrt(3)-1)*cos(260)},{1+0.2*sin(260)}) to[out=270, in=180, looseness=0.5] (0,-1);
        \draw[dashed] ({0+(sqrt(3)-1)*cos(80)},{1+0.2*sin(80)}) to[out=270, in=0, looseness=0.5] (0,-1);
        
        \draw[dashed, rotate around={60:(0,0)}] ({cos(270)+cos(0)*(sqrt(3)-1)},{sin(270)+sin(0)*(sqrt(3)-1)}) arc (0:180:{sqrt(3)-1} and 0.2);
        \draw[rotate around={60:(0,0)}] ({cos(270)+cos(180)*(sqrt(3)-1)},{sin(270)+sin(180)*(sqrt(3)-1)}) arc (180:360:{sqrt(3)-1} and 0.2);

        \draw[dashed, rotate around={-60:(0,0)}] ({cos(270)+cos(0)*(sqrt(3)-1)},{sin(270)+sin(0)*(sqrt(3)-1)}) arc (0:180:{sqrt(3)-1} and 0.2);
        \draw[rotate around={-60:(0,0)}] ({cos(270)+cos(180)*(sqrt(3)-1)},{sin(270)+sin(180)*(sqrt(3)-1)}) arc (180:360:{sqrt(3)-1} and 0.2);
    \end{tikzpicture}
\end{tikzcd}
    \caption{Their distance is uniformly bounded.}
    \label{fig:one_sided_move}
\end{figure}

From this exercise, one deduces that for each marked surface \((h_X,X)\), the map  
\[
\Pants(X) \rightarrow \Pants_n, \quad P_X \mapsto [P_X,h_X,X],
\]  
extended so that each edge is mapped to a geodesic segment, is coarsely Lipschitz. 
As a corollary, the image of \(\Pants(X)\) in \(\Pants_n\) is connected for every marked surface \((h_X,X)\). 
Consequently, we obtain the following result.

% \textcolor{blue}{We have shown in Theorem~3.6, supported by Lemma~3.4, that the inclusion \(\mathcal{P}_n^+ \hookrightarrow \mathcal{P}_n\) is coarsely surjective.}

\begin{lemma} \label{lem:quotient}
    For each $n \geq 3$, the quotient $\Pants_n / \Out(F_n)$ is finite and connected.
\end{lemma}

\begin{proof}
    Each vertex of \(\Pants_n / \Out(F_n)\) corresponds to a pants decomposition of an ``unmarked'' surface. 
    Since the number of such unmarked surfaces is finite and each admits only finitely many pants decompositions, the set of vertices of \(\Pants_n / \Out(F_n)\) is finite.  
    
    Moreover, each pants decomposition consists of finitely many pairs of pants, and therefore finitely many simple closed curves.
    So only finitely many Elementary moves can deform a given pants decomposition (this holds for unmarked surfaces).
    It follows that the number of edges of \(\Pants_n / \Out(F_n)\) is also finite.  
    
    Now choose a base vertex \(P_0 \in \Pants_n\), represented by a pants decomposition of an \(n\)-holed disk. 
    We claim that for every \(P \in \Pants_n\), there exists \(\phi \in \Out(F_n)\) such that there is a path joining \(P\) and \(\phi.P_0\). Let \((P_X,h_X,X)\) be a representative of \(P\).  
    
    If \(X\) is an \(n\)-holed disk, then there exists \(\phi \in \Out(F_n)\) such that both \(\phi.P_0\) and \(P\) are represented by pants decompositions on the same marked surface.
    In this case, a path connecting \(\phi.P_0\) and \(P\), each edge of which is represented by Elementary move I, exists directly.

    If \(X\) is orientable with positive genus, then the quotient of $\Pants(X)$ by the mapping class group is connected via edges corresponding to Elementary move I.
    So there exists a sequence of Elementary move I’s starting from \(P\) such that the resulting vertex is represented by a pants decomposition containing a separating curve that bounds a twice-holed torus, together with two nonseparating curves in that subsurface. 
    From this vertex, one can apply an Elementary move II to obtain adjacency to a vertex corresponding to a surface of strictly lower genus. 
    Iterating this process yields a path from \(P\) to a vertex represented by a pants decomposition of an \(n\)-holed disk. 
    By the previous paragraph, there exists \(\phi \in \Out(F_n)\) such that \(P\) is connected to \(\phi.P_0\) by a path.

    If \(X\) is nonorientable, then there exists a sequence of Elementary move I’s and Elementary move III’s deforming \(P\) to a vertex represented by a pants decomposition whose pair of pants consists of one boundary component and a one-sided curve. 
    From this vertex, one may apply an Elementary move IV to obtain a pants decomposition of strictly lower nonorientable genus. 
    Thus there is a path from \(P\) to a vertex represented by a pants decomposition of lower nonorientable genus. 
    Repeating this procedure inductively, we eventually reach a vertex corresponding to a pants decomposition of an orientable surface. 
    Applying the argument of the previous paragraph, we conclude that there exists \(\phi \in \Out(F_n)\) such that \(P\) is joined to \(\phi.P_0\) by a path in \(\Pants_n\). 
    
    This proves the claim. So every vertex of \(\Pants_n / \Out(F_n)\) is connected by a path to the orbit \(\Out(F_n).P_0\). Hence, \(\Pants_n / \Out(F_n)\) is connected.
\end{proof}

\begin{remark} \label{rem:orientable_pants_quotients}
The middle part of the proof of Lemma \ref{lem:quotient} shows that the quotient \(\Pants_n^+ / \Out(F_n)\) is connected.
Furthermore, since \(\Pants_n^+\) is a subgraph of \(\Pants_n\), the quotient \(\Pants_n^+ / \Out(F_n)\) is also finite.
\end{remark}

We now observe the following fact.

\begin{theorem}
    For each $n \geq 3$, the inclusion map $\Pants_n^+ \hookrightarrow \Pants_n$ is coarsely surjective.
    In particular, we have $\Pants_n^+ \succ \Pants_n$.
\end{theorem}

\begin{proof}
    By Lemma \ref{lem:quotient}, the pants graph \(\Pants_n\) lies within the $D$--neighborhood of \(\Pants_n^+\) where $D$ stands for the diameter of \(\Pants_n / \Out(F_n)\).
    So the inclusion is coarsely surjective.
    Moreover, the inclusion \(\Pants_n^+ \hookrightarrow \Pants_n\) is clearly a \((1,0)\)--coarsely Lipschitz map.
    Therefore, we obtain $\Pants_n^+ \succ \Pants_n$.
\end{proof}

\section{Pants Graphs and Free Splitting Complexes}

% \textcolor{blue}{We have removed some unnecessary sentences from the first paragraph of Section 4 to improve clarity and conciseness.}

Assume $n \geq 3$.  
We say a metric tree equipped with an isometric action of $F_n$ is an \emph{$F_n$-tree}.  
Two $F_n$-trees $T$ and $T'$ are said to be \emph{equivalent} if there exists an isometry $g: T \to T'$ such that $g(w.x) = w.g(x)$ for every $w \in F_n$ and $x \in T$.  
We say an $F_n$-tree $T'$ is a \emph{refinement} of an $F_n$-tree $T$ if there exists an $F_n$--equivariant surjective map from $T'$ to $T$.  
In this case, $T'$ is a refinement of all trees equivalent to $T$, so we can write $[T'] > [T]$.
% A refinement means that $T'$ captures more detailed information about the structure of $T$.
% More precisely, it means that there is an $F_n$-equivariant surjection from $T'$ to $T$ that preserves the action of $F_n$.
% In this case, $T'$ can be seen as a more refined version of $T$, and we write $[T'] > [T]$ to represent this refinement relation.

An $F_n$-tree is called a \emph{free splitting} if every edge stabilizer is trivial.  
The \emph{free splitting complex} $\FS_n$ is constructed as follows.  
The vertices of $\FS_n$ are all equivalence classes of free splittings.  
A collection of $k+1$ vertices $t_0, \dots, t_k$ of $\FS_n$ form a $k$-simplex if $t_0 < t_1 < \dots < t_k$.

% \textcolor{red}{Page 8, last sentence: If the arc $\alpha$ is non-separating, you don’t get a wedge of surfaces, but rather a single surface with two points of its boundary glued together. And in any case, it is more standard to define the splitting as the dual to the preimage of $\alpha$ on the universal cover.}

% \textcolor{blue}{We have replaced the sentence with ``Then two points of the boundary of \(X\) are identified.'' In addition, we have incorporated the standard definition suggested by the referee.}

% Let $h_X: R_n \to X$ be a homotopy equivalence and $\alpha \subset X$ be an essential arc.
% Homotope $X$ by shrinking $\alpha$ to a point.
% Then two points of the boundary of $X$ are identified.
% By the van Kampen theorem, this corresponds to a free splitting of $F_n$.
% This free splitting can be seen as a dual tree of the preimage of $\alpha$ on the universal cover of $X$.
% Let us call such a free splitting, denoted by $T_\alpha$, the \emph{free splitting corresponding to $\alpha$}.

Let $h_X: R_n \to X$ be a homotopy equivalence and let $\alpha \subset X$ be an essential arc. 
Homotope $X$ by shrinking $\alpha$ to a point, so that the two boundary points of $\alpha$ become identified.
By van Kampen's theorem, this operation yields a free splitting of $F_n$.
Equivalently, this splitting may be viewed as the dual tree to the preimage of $\alpha$ in the universal cover of $X$.
We refer to this free splitting, denoted $T_\alpha$, as the \emph{free splitting corresponding to $\alpha$}.

\begin{remark}
    For each marked surface $(h_X, X)$, the barycentric subdivision of the arc complex of $X$ is embedded into the free splitting complex via $\alpha \mapsto T_\alpha$.
\end{remark}

% \textcolor{red}{Page 9, second paragraph after Remark 4.1: When $\alpha$ and $\alpha_1$ intersect, it looks to me like that have distance at most 2 in the arc graph (which is a bit better than 3 as you state). Also, the last sentence says something “has diameter six”, and I think you only want claim an upper bound (so you should add “at most”).}

% \textcolor{blue}{Yes, you are right. I have corrected it accordingly.}

For each vertex $P = [P_X, h_X, X]$ of $\Pants_n$, let us define $\Psi(P_X, h_X, X)$ by the free splitting corresponding to $\alpha$ for some arc $\alpha$ disjoint from $P_X$.
Note that $\Psi(P_X, h_X, X)$ is not well-defined because a different choice of an arc in $X$ induces a different free splitting.
Nevertheless, $\Psi(P_X, h_X, X)$ is coarsely well-defined.
That is, if $\alpha'$ is another arc disjoint from $P_X$, the free splitting $T$ corresponding to $\alpha'$ is still close to $\Psi(P_X, h_X, X)$ in $\FS(F_n)$.

The precise proof of the above assertion is the following.
If $\alpha$ and $\alpha'$ are disjoint, then $\Psi(P_X, h_X, X)$ and $T$ have a common refinement, so $d_\FS(\Psi(P_X, h_X, X), T)$ is at most two.  
Otherwise, $\alpha$ and $\alpha'$ are contained in the same pair of pants in $P_X$.  
If the endpoints of $\alpha$ and $\alpha'$ are contained in the same boundary curve, then the free splittings corresponding to $\alpha$ and $\alpha'$ are equal.  
If $\alpha$ and $\alpha'$ intersect, they are still contained in the same pair of pants but with their endpoints on different boundary curves.  
In this case, the distance between their corresponding free splittings is at most six (because the distance in the arc complex is three).  
Therefore, the set of all free splittings corresponding to the arcs disjoint from $(P_X, h_X, X)$ has diameter at most six.

Define a map $\Psi_n: \Pants_n \to \FS_n$ by $\Psi_n(P) = \Psi(P_X, h_X, X)$ for each $P = [P_X, h_X, X]$, and extend it by sending each edge to a geodesic.

\begin{theorem} \label{thm:freelip}
    % \textcolor{blue}{We have clarified the statement of Theorem~4.2 to make it more precise and easier to understand.}
    
    For each $n \geq 3$, the map $\Psi_n$ is $(16, 0)$--coarsely Lipschitz with 
    \[
        d_\FS(\Psi_n(\phi.P), \phi.\Psi_n(P)) \leq 6
    \]
    for all $P \in \Pants_n$ and $\phi \in \Out(F_n)$.
\end{theorem}

\begin{proof}
    Let $P$ and $Q$ be adjacent vertices of $\Pants_n$.
    Suppose $Q$ is obtained from $P$ by an Elementary move I.  
    Write $P = [P_X, h_X, X]$ and $Q = [Q_X, h_X, X]$.  
    Let $\alpha_P$ and $\alpha_Q$ be arcs of $X$ such that $\Psi_n(P)$ and $\Psi_n(Q)$ are the free splittings corresponding to $\alpha_P$ and $\alpha_Q$, respectively.  
    If $\alpha_P$ and $\alpha_Q$ do not fill $X$, then the distance between $\alpha_P$ and $\alpha_Q$ in the arc complex is at most two; in this case, we have $d_\FS(\Psi_n(P), \Psi_n(Q)) \leq 4$.  
    Assume $\alpha_P$ and $\alpha_Q$ fill $X$.  
    Then $P$ and $Q$ contain curves that fill $X$, which implies $X$ is homeomorphic to a four-holed sphere.  
    Note that $\alpha_P$ and $\alpha_Q$ intersect twice.  
    Thus, $\alpha_P$ and $\alpha_Q$ are at distance at most two in the arc complex, so we have $d_\FS(\Psi_n(P), \Psi_n(Q)) \leq 4$.

    Suppose $Q$ is obtained from $P$ by an Elementary move II.  
    Without loss of generality, assume that the number of boundary curves of $Q$ is less than that of $P$.  
    Write $Q = [Q_Y, h_Y, Y]$.  
    Let $\alpha_P' \subset X$ be the arc collapsed in the move, and let $\alpha_Q' \subset Y$ be the arc collapsed in the inverse move.
    The free splittings corresponding to $\alpha_P'$ and $\alpha_Q'$ are of the form $A * \langle b \rangle$ and $\langle A, bab^{-1} \rangle *_b$, respectively.  
    The distance between these free splittings is at most four.  
    Since $\alpha_P'$ is disjoint from $P_X$, the free splittings corresponding to $\alpha_P$ and $\alpha_P'$ have distance $\leq 6$.  
    Similarly, the free splittings corresponding to $\alpha_Q$ and $\alpha_Q'$ have distance at most six.  
    Thus, the distance between $\Psi_n(P)$ and $\Psi_n(Q)$ is at most sixteen.

    Suppose that $Q$ is obtained from $P$ by an Elementary move III. 
    Then the arc $\alpha_Q$ is disjoint from the pants decomposition $P_X$. 
    Likewise, $\alpha_P$ is disjoint from $Q_X$.
    It follows that the vertices $\Psi_n(P)$ and $\Psi_n(Q)$ lie at distance at most six.
    
    Suppose that $Q$ is obtained from $P$ by an Elementary move IV. 
    Then there exist arcs $\alpha_P'' \subset X$ and $\alpha_Q'' \subset Y$ such that the associated free splittings $T_{\alpha_P''}$ and $T_{\alpha_Q''}$ are equivalent. 
    Consequently, the vertices $\Psi_n(P)$ and $\Psi_n(Q)$ lie at distance at most twelve.

    For a geodesic sequence $P_0, \dots, P_k$ in $\Pants_n$, we have
    \[
        d_\FS(\Psi_n(P_0), \Psi_n(P_k)) \leq \sum_{i=1}^k d_\FS(\Psi_n(P_{i-1}), \Psi_n(P_i)) \leq 16k = 16\, d_\Pants(P_0, P_k).
    \]
    Therefore, $\Psi_n$ is a $(16, 0)$--coarsely Lipschitz map.

    % \textcolor{blue}{We have added a proof showing that \(\Psi_n\) is coarsely \(\operatorname{Out}(F_n)\)--equivariant in the proof of Theorem 4.2.}

    We now verify that \(\Psi_n\) is coarsely \(\Out(F_n)\)--equivariant.
    Let \(P = [P_X, h_X, X] \in \Pants_n\) and let \(\phi \in \Out(F_n)\) with representative \(f : R_n \to R_n\).
    Let \(\alpha_P\) be an arc such that the corresponding free splitting is \(\Psi_n(P)=T_{\alpha_P}\).
    Then \(\phi.\Psi_n(P)\) is the free splitting corresponding to \(\alpha_P\) of the marked surface \((h_X \circ f^{-1},X)\).
    
    On the other hand, the vertex \(\Psi_n(P)\) arises from the free splitting corresponding to some arc that is disjoint from the pants decomposition \((P_X, h_X \circ f^{-1}, X)\).
    By the preceding discussion, these two splittings lie at distance at most six.
    Therefore \(\Psi_n\) is \(\Out(F_n)\)--equivariant up to a uniformly bounded error, i.e.\ it is coarsely \(\Out(F_n)\)--equivariant.
\end{proof}

% \textcolor{blue}{We have added Corollary 4.3, which shows that \(\Psi_n\) is coarsely surjective.}

Note that the number of $\Out(F_n)$--orbits of vertices of $\mathcal{FS}_n$ is finite.
On the other hand, every one-edge free splittings are contained in the six--neighborhood of the image of $\Psi_n$.
Combining these two facts, we can conclude the following.

\begin{corollary} \label{cor:coarse_surj}
    For each $n \geq 3$, the map $\Psi_n$ is coarsely surjective.
    In particular, we have $\Pants_n \succ \FS_n$.
\end{corollary}

% \textcolor{blue}{We have removed Remark 4.3 from the original version.}

% \begin{remark}
%     Shenitzer \cite{MR69174}, Swarup \cite{MR825183}, and Stallings \cite{MR1105341} showed that if a cyclic splitting of $F_n$ has a nontrivial edge stabilizer, then there exists an incident pair $(v, e)$ such that the edge group $G_e$ is nontrivial and the vertex group $G_v$ is decomposed into $G_e * A$.
% \end{remark}

Handel and Mosher \cite[Theorem 1.1(1)]{MR4009387} characterized the loxodromic elements of $\Out(F_n)$ in terms of attracting laminations: an element of $\Out(F_n)$ is loxodromic in the free splitting complex if and only if it has a filling attracting lamination.  
Theorem \ref{thm:freelip} gives the following result.

% \textcolor{red}{Page 10, Corollary 4.4: Please explain why a element without a filling attracting lamination must have bounded orbits in $\mathcal{P}_n$.}

% \textcolor{blue}{We have added the proof of Corollary 4.4.}

\begin{corollary}
    An element of $\Out(F_n)$ has an unbounded orbit in both $\Pants_n$ and $\Pants_n^+$ if it has a filling attracting lamination.
\end{corollary}

\begin{proof}
    Let $\phi \in \Out(F_n)$ be an element that has a filling attracting lamination.
    By Handel and Mosher \cite{MR4009387}, $\phi$ acts loxodromically on $\FS_n$, that is,
    \[
        \lim_{m \to \infty}\frac{d_\FS(T, \phi^m.T)}{m} > 0
    \]
    for some $T \in \FS_n$.
    By Corollary \ref{cor:coarse_surj}, the intersection of $\Psi_n(\Pants_n)$ and the $C$--neighborhood of $T$ is nonemtpy for some $C = C(n)$.
    From this intersection, take an element $T'$ and choose $P$ such that $\Psi_n(P) = T'$.
    Then we have
    \begin{align*}
        d_\Pants(P, \phi^m.P) &\geq \frac{1}{16}\, d_\FS(\Psi_n(P), \Psi_n(\phi^m.P)) \\
        &\geq \frac{1}{16}(d_\FS(T, \phi^m.T) - d_\FS(T, \Psi_n(P)) - d_\FS(\Psi_n(\phi^m.P), \phi^m.\Psi_n(P)) \\
        &\quad\, - d_\FS(\phi^m.\Psi_n(P), \phi^m.T)) \\
        &\geq \frac{1}{16}(d_\FS(T, \phi^m.T) - 2C - 6) \\
        &= \frac{1}{16}\, d_\FS(T, \phi^m.T) - \frac{C+3}{8},
    \end{align*}
    which tends to infinity as $m \to \infty$.
    Hence, $\phi$ has an unbounded orbit in $\Pants_n$.
\end{proof}

The fact that free splitting complexes are unbounded implies the following.

\begin{corollary}
    For each $n \geq 2$, both $\Pants_n$ and $\Pants_n^+$ are unbounded.
\end{corollary}

\section{Pants Graphs and Cayley Graphs}

Fix $n \geq 3$.  
It is well-known that $\Out(F_n)$ is finitely generated.  
Thus, we can define a word metric on $\Out(F_n)$, which turns $\Out(F_n)$ into a metric space.  
Since any two word metrics on a finitely generated group are quasi-isometric, this implies that the coarse Lipschitz property of a map from $\Out(F_n)$ to a metric space does not depend on the specific word metric chosen, but rather on $\Out(F_n)$ itself.  
In this section, we show that the orbit map of $\Out(F_n)$ on the orientable pants graph is coarsely Lipschitz.

% \textcolor{blue}{We have removed Section 5.1 and incorporated Remark 5.1 into the main text of the revised version.}

% \subsection{Cayley graphs and pants graphs}

As mentioned above, the group $\Out(F_n)$ is finitely generated. 
This was first established by Nielsen \cite{MR1512188} in his work on automorphism groups of free groups, 
where he introduced four types of automorphisms that generate $\Aut(F_n)$.
Neumann \cite{MR1512806} showed that $\Aut(F_n)$ can be generated by two elements for all $n \geq 4$.
In the present context, we will use the fact that $\Out(F_n)$ is generated by transvections and inversions.

Let us recall the definition of transvections and inversion.
First define the rose $R_n$ as q quotient space $[0, n] / \{0, 1, \dots, n\}$ that can be seen as a graph with $n$ petals with the quotient map $q: [0, n] \to R_n$.
For each $i = 1, \dots, n$, the segment $q([i - 1, i])$ is referred to as the \emph{$i$-th petal of $R_n$}.  
Now, for distinct $i, j \in \{1, \dots, n\}$, we define two types of transvections:

\begin{itemize}
    \item The \emph{right transvection} $r_{ij}$ is the homotopy class of a map $R_{ij}: R_n \to R_n$, defined by:
    \[
    R_{ij}(t) = 
    \begin{cases} 
    t & \text{if}~ t \in R_n - q([i-1, i]), \\
    q(2q^{-1}(t) - i + 1) & \text{if}~ t \in q([i-1, i-\frac{1}{2}]), \\
    q(2q^{-1}(t) - 2i + j) & \text{if}~ t \in q([i-\frac{1}{2}, i]).
    \end{cases}
    \]

    \item The \emph{left transvection} $\ell_{ji}$ is the homotopy class of a map $L_{ji}: R_n \to R_n$, defined by:
    \[
    L_{ji}(t) = 
    \begin{cases} 
    t & \text{if}~ t \in R_n - q([i-1, i]), \\
    q(2q^{-1}(t) - 2i + j + 1) & \text{if}~ t \in q([i-1, i-\frac{1}{2}]), \\
    q(2q^{-1}(t) - i) & \text{if}~ t \in q([i-\frac{1}{2}, i]).
    \end{cases}
    \]
\end{itemize}
The \emph{inversion} $\iota_{i}$ is the homotopy class of a map $I_i: R_n \to R_n$, defined by:
\[
I_i(t) = 
\begin{cases} 
t & \text{if}~ t \in R_n - q([i-1, i]), \\
q(2i-1 - q^{-1}(t)) & \text{if}~ t \in q([i-1, i]).
\end{cases}
\]

\begin{lemma} \label{lem:transvection}
    For distinct $i, j \in \{1, \dots, n\}$, any element $t \in \{ r_{ij}, r_{ij}^{-1}, \ell_{ji}, \ell_{ji}^{-1} \}$ fixes some vertex of $\Pants^+_n$.
\end{lemma}

\begin{proof}
    Let $X$ be an oriented $(n-1)$-holed torus.  
    Define a homotopy equivalence $h_X: R_n \to X$ as follows:  
    a regular neighborhood of the union of the images of the $i$-th and $j$-th petals forms a once-punctured torus $Y \subset X$.
    For each $\ell \in \{1, \dots, n\} \setminus \{i, j\}$, the image of the $\ell$-th petal is homotopic to a boundary curve of $X$.

    Let $\gamma$ be a simple closed curve homotopic to the $j$-th petal.  
    The right-handed Dehn twist $T_\gamma$ along $\gamma$ corresponds to $t \in \{ r_{ij}, r_{ij}^{-1}, \ell_{ji}, \ell_{ji}^{-1} \}$, where the specific transvection $t$ is determined by the orientation of the $i$-th and $j$-th petals.

    Let $P_X$ be a pants decomposition of $X$ that includes the curves $\gamma \cup \partial Y$.
    Since $T_\gamma$ is a Dehn twist along $\gamma$, it follows that $T_\gamma(P_X) = P_X$.  
    Moreover, since $(P_X, h_X, X) = (T_\gamma. P_X, h_X, X)$ is equivalent to $(P_X, h_X \circ t^{-1}, X)$, we conclude that $P = t.P$.  
    Therefore, the transvection $t$ fixes $P$.
\end{proof}

\begin{lemma} \label{lem:inversion}
    For each $i \in \{1, \dots, n\}$, the inversion $\iota_i$ preserves some $4$-cycle of $\Pants^+_n$.
\end{lemma}

\begin{proof}
    Let $X$ be an $n$-holed disk.  
    Define a topological embedding $h_X: R_n \to X$ such that the image $h_X(R_n)$ is a deformation retract of $X$.  
    Let $\beta$ be the boundary curve of $X$ that is not isotopic to any petal of $R_n$.  
    Additionally, let $\beta_i$ be the boundary curve isotopic to the $i$-th petal.

    Now, choose an arc $\alpha$ that joins $\beta$ and $\beta_i$ and intersects $h_X(R_n)$ only along the $i$-th petal.  
    Let $\gamma$ be a boundary of a regular neighborhood of $\beta \cup \alpha \cup \beta_i$.  
    Next, choose another boundary curve $\beta'$ and let $\delta$ be a simple closed curve such that $\gamma \cup \beta' \cup \delta$ forms a pair of pants.  

    Let $P_X$ be a pants decomposition of $X$ that includes the curves $\gamma \cup \delta$, and define the corresponding pants decomposition $P := [P_X, h_X, X]$.

    \begin{figure}[ht]
        \centering
        \begin{tikzcd}
    \begin{tikzpicture}[thick]
        \draw ({-2*sqrt(3)/2+cos(-60)},{1+sin(-60)}) arc (-60:0:1);
        \draw ({2*sqrt(3)/2+cos(180)},{1+sin(180)}) arc (180:300:1);
        \draw ({cos(-20)},{-2+sin(-20)}) arc (-20:120:1);
        \draw ({2*sqrt(3)+cos(120)},{-2+sin(120)}) arc (120:{180+20}:1);

        \draw[very thick, red, -{Stealth[length=8]}] ({-(sqrt(3)-0.98)},0.8) arc (-180:-80:{sqrt(3)-0.98} and 0.2);
        \draw[very thick, red] (0,0.6) arc (-90:0:{sqrt(3)-0.98} and 0.2);
        \draw[very thick, dashed, red] ({sqrt(3)-0.98},0.8) arc (0:180:{sqrt(3)-0.98} and 0.2);
        \draw (0,1) ellipse ({sqrt(3)-1} and 0.2);
        \draw[rotate around={120:({cos(210)},{sin(210)})}] ({cos(210)},{sin(210)}) ellipse ({sqrt(3)-1} and 0.2);
        
        \draw[rotate around={240:({cos(330)},{sin(330)})}] ({cos(330)+cos(0)*(sqrt(3)-1)},{sin(330)+sin(0)*(sqrt(3)-1)}) arc (0:180:{sqrt(3)-1} and 0.2);
        \draw[dashed, rotate around={240:({cos(330)},{sin(330)})}] ({cos(330)+cos(180)*(sqrt(3)-1)},{sin(330)+sin(180)*(sqrt(3)-1)}) arc (180:360:{sqrt(3)-1} and 0.2);
        
        \draw[rotate around={-60:({sqrt(3)+cos(30)},{-1+sin(30)})}] ({sqrt(3)+cos(30)},{-1+sin(30)}) ellipse ({sqrt(3)-1} and 0.2);
        
        \draw[dashed] ({sqrt(3)+cos(270)+cos(0)*(sqrt(3)-1)},{-1+sin(270)+sin(0)*(sqrt(3)-1)}) arc (0:180:{sqrt(3)-1} and 0.2);
        \draw ({sqrt(3)+cos(270)+cos(180)*(sqrt(3)-1)},{-1+sin(270)+sin(180)*(sqrt(3)-1)}) arc (180:360:{sqrt(3)-1} and 0.2);
    \end{tikzpicture}
    \ar[rr,leftrightarrow] \ar[dd,leftrightarrow]
    &&
    \begin{tikzpicture}[thick]
        \draw (0,0) circle (0.5);
        \draw[very thick, red, -{Stealth[length=8]}] (1,0) arc (0:100:1);
        \draw[very thick, red] (0,1) arc (90:360:1);
        
        \draw ({sqrt(3)-1+(0.5+(sqrt(3)-1))*cos(-atan(2/(2/((sqrt(3)-1)+0.5)-((sqrt(3)-1)+0.5)/2)))},{(0.5+(sqrt(3)-1))*sin(-atan(2/(2/((sqrt(3)-1)+0.5)-((sqrt(3)-1)+0.5)/2)))}) arc ({-atan(2/(2/((sqrt(3)-1)+0.5)-((sqrt(3)-1)+0.5)/2))}:{atan(2/(2/((sqrt(3)-1)+0.5)-((sqrt(3)-1)+0.5)/2))}:{0.5+(sqrt(3)-1)});
        \draw ({-(sqrt(3)-1)+(0.5+(sqrt(3)-1))*cos(180-atan(2/(2/((sqrt(3)-1)+0.5)-((sqrt(3)-1)+0.5)/2)))},{(0.5+(sqrt(3)-1))*sin(180-atan(2/(2/((sqrt(3)-1)+0.5)-((sqrt(3)-1)+0.5)/2)))}) arc ({180-atan(2/(2/((sqrt(3)-1)+0.5)-((sqrt(3)-1)+0.5)/2))}:{180+atan(2/(2/((sqrt(3)-1)+0.5)-((sqrt(3)-1)+0.5)/2))}:{0.5+(sqrt(3)-1)});

        \draw ({(sqrt(3)-1)+2/((sqrt(3)-1)+0.5)-((sqrt(3)-1)+0.5)/2+(2/((sqrt(3)-1)+0.5)-((sqrt(3)-1)+0.5)/2)*cos(180-atan(2/(2/((sqrt(3)-1)+0.5)-((sqrt(3)-1)+0.5)/2)))},{-2+(2/((sqrt(3)-1)+0.5)-((sqrt(3)-1)+0.5)/2)*sin(180-atan(2/(2/((sqrt(3)-1)+0.5)-((sqrt(3)-1)+0.5)/2)))}) arc ({180-atan(2/(2/((sqrt(3)-1)+0.5)-((sqrt(3)-1)+0.5)/2))}:200:{2/((sqrt(3)-1)+0.5)-((sqrt(3)-1)+0.5)/2});
        \draw ({-((sqrt(3)-1)+2/((sqrt(3)-1)+0.5)-((sqrt(3)-1)+0.5)/2)+(2/((sqrt(3)-1)+0.5)-((sqrt(3)-1)+0.5)/2)*cos(-20)},{-2+(2/((sqrt(3)-1)+0.5)-((sqrt(3)-1)+0.5)/2)*sin(-20)}) arc (-20:{atan(2/(2/((sqrt(3)-1)+0.5)-((sqrt(3)-1)+0.5)/2))}:{2/((sqrt(3)-1)+0.5)-((sqrt(3)-1)+0.5)/2});
        \draw ({(sqrt(3)-1)+2/((sqrt(3)-1)+0.5)-((sqrt(3)-1)+0.5)/2+(2/((sqrt(3)-1)+0.5)-((sqrt(3)-1)+0.5)/2)*cos(180)},{2+(2/((sqrt(3)-1)+0.5)-((sqrt(3)-1)+0.5)/2)*sin(180)}) arc (180:{180+atan(2/(2/((sqrt(3)-1)+0.5)-((sqrt(3)-1)+0.5)/2))}:{2/((sqrt(3)-1)+0.5)-((sqrt(3)-1)+0.5)/2});
        \draw ({-((sqrt(3)-1)+2/((sqrt(3)-1)+0.5)-((sqrt(3)-1)+0.5)/2)+(2/((sqrt(3)-1)+0.5)-((sqrt(3)-1)+0.5)/2)*cos(0)},{2+(2/((sqrt(3)-1)+0.5)-((sqrt(3)-1)+0.5)/2)*sin(0)}) arc (0:{-atan(2/(2/((sqrt(3)-1)+0.5)-((sqrt(3)-1)+0.5)/2))}:{2/((sqrt(3)-1)+0.5)-((sqrt(3)-1)+0.5)/2});

        \draw ({sqrt(3)-1},-2) arc (0:-180:{sqrt(3)-1} and 0.2);
        \draw[dashed] ({sqrt(3)-1},-2) arc (0:180:{sqrt(3)-1} and 0.2);
        \draw ({0.5+2*(sqrt(3)-1)},0) arc (0:-180:{sqrt(3)-1} and 0.2);
        \draw[dashed] ({0.5+2*(sqrt(3)-1)},0) arc (0:180:{sqrt(3)-1} and 0.2);
        \draw (-0.5,0) arc (0:-180:{sqrt(3)-1} and 0.2);
        \draw[dashed] (-0.5,0) arc (0:180:{sqrt(3)-1} and 0.2);
        \draw (0,2) ellipse ({sqrt(3)-1} and 0.2);
    \end{tikzpicture}
    \ar[dd,leftrightarrow]
    \\
    &
    \begin{tikzpicture}[thick]            
        \draw ({2*sqrt(3)/2+cos(240)},{1+sin(240)}) arc (240:300:1);
        \draw ({cos(-20)},{-2+sin(-20)}) arc (-20:60:1);
        \draw ({2*sqrt(3)+cos(120)},{-2+sin(120)}) arc (120:{180+20}:1);

        \draw[rotate around={240:({cos(330)},{sin(330)})}] ({cos(330)+cos(0)*(sqrt(3)-1)},{sin(330)+sin(0)*(sqrt(3)-1)}) arc (0:180:{sqrt(3)-1} and 0.2);
        \draw[rotate around={240:({cos(330)},{sin(330)})}] ({cos(330)+cos(180)*(sqrt(3)-1)},{sin(330)+sin(180)*(sqrt(3)-1)}) arc (180:360:{sqrt(3)-1} and 0.2);
        
        \draw[rotate around={-60:({sqrt(3)+cos(30)},{-1+sin(30)})}] ({sqrt(3)+cos(30)},{-1+sin(30)}) ellipse ({sqrt(3)-1} and 0.2);
        
        \draw[dashed] ({sqrt(3)+cos(270)+cos(0)*(sqrt(3)-1)},{-1+sin(270)+sin(0)*(sqrt(3)-1)}) arc (0:180:{sqrt(3)-1} and 0.2);
        \draw ({sqrt(3)+cos(270)+cos(180)*(sqrt(3)-1)},{-1+sin(270)+sin(180)*(sqrt(3)-1)}) arc (180:360:{sqrt(3)-1} and 0.2);

        \draw[very thick, red, rotate around={60:({cos(330)},{sin(330)})}] ({cos(330)+cos(0)*(sqrt(3)-1)+0.5},{sin(330)+sin(0)*(sqrt(3)-1)}) circle (0.5);
    \end{tikzpicture}
    \ar[dr,leftrightarrow,dashed] \ar[ur,leftrightarrow,dashed] \ar[dl,leftrightarrow,dashed] \ar[ul,leftrightarrow,dashed]
    &
    \\
    \begin{tikzpicture}[thick]
        \draw (0,0) circle (0.5);
        \draw[very thick, red, -{Stealth[length=8]}] (1,0) arc (0:-280:1);
        \draw[very thick, red] (0,1) arc (-270:-360:1);
        
        \draw ({sqrt(3)-1+(0.5+(sqrt(3)-1))*cos(-atan(2/(2/((sqrt(3)-1)+0.5)-((sqrt(3)-1)+0.5)/2)))},{(0.5+(sqrt(3)-1))*sin(-atan(2/(2/((sqrt(3)-1)+0.5)-((sqrt(3)-1)+0.5)/2)))}) arc ({-atan(2/(2/((sqrt(3)-1)+0.5)-((sqrt(3)-1)+0.5)/2))}:{atan(2/(2/((sqrt(3)-1)+0.5)-((sqrt(3)-1)+0.5)/2))}:{0.5+(sqrt(3)-1)});
        \draw ({-(sqrt(3)-1)+(0.5+(sqrt(3)-1))*cos(180-atan(2/(2/((sqrt(3)-1)+0.5)-((sqrt(3)-1)+0.5)/2)))},{(0.5+(sqrt(3)-1))*sin(180-atan(2/(2/((sqrt(3)-1)+0.5)-((sqrt(3)-1)+0.5)/2)))}) arc ({180-atan(2/(2/((sqrt(3)-1)+0.5)-((sqrt(3)-1)+0.5)/2))}:{180+atan(2/(2/((sqrt(3)-1)+0.5)-((sqrt(3)-1)+0.5)/2))}:{0.5+(sqrt(3)-1)});

        \draw ({(sqrt(3)-1)+2/((sqrt(3)-1)+0.5)-((sqrt(3)-1)+0.5)/2+(2/((sqrt(3)-1)+0.5)-((sqrt(3)-1)+0.5)/2)*cos(180-atan(2/(2/((sqrt(3)-1)+0.5)-((sqrt(3)-1)+0.5)/2)))},{-2+(2/((sqrt(3)-1)+0.5)-((sqrt(3)-1)+0.5)/2)*sin(180-atan(2/(2/((sqrt(3)-1)+0.5)-((sqrt(3)-1)+0.5)/2)))}) arc ({180-atan(2/(2/((sqrt(3)-1)+0.5)-((sqrt(3)-1)+0.5)/2))}:200:{2/((sqrt(3)-1)+0.5)-((sqrt(3)-1)+0.5)/2});
        \draw ({-((sqrt(3)-1)+2/((sqrt(3)-1)+0.5)-((sqrt(3)-1)+0.5)/2)+(2/((sqrt(3)-1)+0.5)-((sqrt(3)-1)+0.5)/2)*cos(-20)},{-2+(2/((sqrt(3)-1)+0.5)-((sqrt(3)-1)+0.5)/2)*sin(-20)}) arc (-20:{atan(2/(2/((sqrt(3)-1)+0.5)-((sqrt(3)-1)+0.5)/2))}:{2/((sqrt(3)-1)+0.5)-((sqrt(3)-1)+0.5)/2});
        \draw ({(sqrt(3)-1)+2/((sqrt(3)-1)+0.5)-((sqrt(3)-1)+0.5)/2+(2/((sqrt(3)-1)+0.5)-((sqrt(3)-1)+0.5)/2)*cos(180)},{2+(2/((sqrt(3)-1)+0.5)-((sqrt(3)-1)+0.5)/2)*sin(180)}) arc (180:{180+atan(2/(2/((sqrt(3)-1)+0.5)-((sqrt(3)-1)+0.5)/2))}:{2/((sqrt(3)-1)+0.5)-((sqrt(3)-1)+0.5)/2});
        \draw ({-((sqrt(3)-1)+2/((sqrt(3)-1)+0.5)-((sqrt(3)-1)+0.5)/2)+(2/((sqrt(3)-1)+0.5)-((sqrt(3)-1)+0.5)/2)*cos(0)},{2+(2/((sqrt(3)-1)+0.5)-((sqrt(3)-1)+0.5)/2)*sin(0)}) arc (0:{-atan(2/(2/((sqrt(3)-1)+0.5)-((sqrt(3)-1)+0.5)/2))}:{2/((sqrt(3)-1)+0.5)-((sqrt(3)-1)+0.5)/2});

        \draw ({sqrt(3)-1},-2) arc (0:-180:{sqrt(3)-1} and 0.2);
        \draw[dashed] ({sqrt(3)-1},-2) arc (0:180:{sqrt(3)-1} and 0.2);
        \draw ({0.5+2*(sqrt(3)-1)},0) arc (0:-180:{sqrt(3)-1} and 0.2);
        \draw[dashed] ({0.5+2*(sqrt(3)-1)},0) arc (0:180:{sqrt(3)-1} and 0.2);
        \draw (-0.5,0) arc (0:-180:{sqrt(3)-1} and 0.2);
        \draw[dashed] (-0.5,0) arc (0:180:{sqrt(3)-1} and 0.2);
        \draw (0,2) ellipse ({sqrt(3)-1} and 0.2);
    \end{tikzpicture}
    \ar[rr,leftrightarrow]
    &&
    \begin{tikzpicture}[thick]
        \draw ({-2*sqrt(3)/2+cos(-60)},{1+sin(-60)}) arc (-60:0:1);
        \draw ({2*sqrt(3)/2+cos(180)},{1+sin(180)}) arc (180:300:1);
        \draw ({cos(-20)},{-2+sin(-20)}) arc (-20:120:1);
        \draw ({2*sqrt(3)+cos(120)},{-2+sin(120)}) arc (120:{180+20}:1);

        \draw[very thick, red] ({-(sqrt(3)-0.98)},0.8) arc (-180:-90:{sqrt(3)-0.98} and 0.2);
        \draw[very thick, red, -{Stealth[length=8]}] ({sqrt(3)-0.98},0.8) arc (0:-100:{sqrt(3)-0.98} and 0.2);
        \draw[very thick, dashed, red] ({sqrt(3)-0.98},0.8) arc (0:180:{sqrt(3)-0.98} and 0.2);
        \draw (0,1) ellipse ({sqrt(3)-1} and 0.2);
        \draw[rotate around={120:({cos(210)},{sin(210)})}] ({cos(210)},{sin(210)}) ellipse ({sqrt(3)-1} and 0.2);
        
        \draw[rotate around={240:({cos(330)},{sin(330)})}] ({cos(330)+cos(0)*(sqrt(3)-1)},{sin(330)+sin(0)*(sqrt(3)-1)}) arc (0:180:{sqrt(3)-1} and 0.2);
        \draw[dashed, rotate around={240:({cos(330)},{sin(330)})}] ({cos(330)+cos(180)*(sqrt(3)-1)},{sin(330)+sin(180)*(sqrt(3)-1)}) arc (180:360:{sqrt(3)-1} and 0.2);
        
        \draw[rotate around={-60:({sqrt(3)+cos(30)},{-1+sin(30)})}] ({sqrt(3)+cos(30)},{-1+sin(30)}) ellipse ({sqrt(3)-1} and 0.2);
        
        \draw[dashed] ({sqrt(3)+cos(270)+cos(0)*(sqrt(3)-1)},{-1+sin(270)+sin(0)*(sqrt(3)-1)}) arc (0:180:{sqrt(3)-1} and 0.2);
        \draw ({sqrt(3)+cos(270)+cos(180)*(sqrt(3)-1)},{-1+sin(270)+sin(180)*(sqrt(3)-1)}) arc (180:360:{sqrt(3)-1} and 0.2);
    \end{tikzpicture}
\end{tikzcd}
        \caption{A cycle preserved by $\iota_i$} %\textcolor{blue}{We have replaced Figure~8 from the original version with a new figure created in the \texttt{tikz} environment.}}
        \label{fig:cycle_for_inversion}
    \end{figure}

    The pants decomposition $P$ and its image under the inversion, $\iota_i.P = [P_X, h_X \circ \iota_i^{-1}, X)$, are joined by a path of length two, as shown in Figure \ref{fig:cycle_for_inversion}.
    Precisely, $P$ and $\iota_i.P$ are homotopic to the common wedge of a circle and a surface as in the middle of Figure \ref{fig:cycle_for_inversion}.
    This implies that $P$ and $\iota_i.P$ are incident to a common vertex $Q$.
    
    Furthermore, since $\iota^2$ is the identity, $\iota_i.Q$ also lies between $P$ and $\iota_i.P$.
    Hence, $\iota_i$ preserves the cycle $P, Q, \iota_i.P, \iota_i.Q$.
\end{proof}

Based on the proofs of Lemma \ref{lem:transvection} and Lemma \ref{lem:inversion}, we can explicitly identify a cycle in the pants graph $\Pants^+_n$ that is preserved by either a transvection or an inversion. Let $S_n$ denote the set of all transvections and inversions. This set $S_n$ forms a finite generating set for the outer automorphism group $\Out(F_n)$.

% \textcolor{red}{Page 11, Lemma 5.4: Do you want to require that your connected subgraph S be bounded?}

% \textcolor{blue}{That is a good point. In fact, since it is already known that the action of \(\operatorname{Out}(F_n)\) on \(\mathcal{P}_n^+\) is cocompact, it is not necessary to prove that \(\mathcal{S}\) is bounded. Nevertheless, we believe there is no harm in noting explicitly that the subgraph \(\mathcal{S}\) is bounded. We thank the referee for this helpful comment.}

\begin{lemma} \label{lem:connected_S}
    There exists a bounded connected subgraph $\mathcal{S}$ of $\Pants^+_n$ such that $\phi.\mathcal{S}$ and $\mathcal{S}$ intersect for each $\phi \in S_n$.
\end{lemma}

\begin{proof}
    Let $X$ be an $n$-holed disk, and let $h_X: R_n \to X$ be a topological embedding which is also a homotopy equivalence.  
    Let $\beta_i$ denote the boundary curve of $X$ that is isotopic to the $i$-th petal, and let $\beta_0$ denote the boundary curve that is not isotopic to any petal.  

    For distinct $1 \leq i < j \leq n$, let us take a simple closed curve $\gamma_{ij}$ such that $\gamma_{ij}$ is disjoint from the $\ell$-th petal for all $\ell \in \{1, \dots, n\} \setminus \{i, j\}$, and such that $\gamma_{ij} \cup \beta_i \cup \beta_j$ forms a pair of pants.  
    Let $P_{ij}$ denote a pants decomposition of $X$ that contains the curve $\gamma_{ij}$.  
    Since $\Pants(X)$ is connected and the map $H_X: \Pants(X) \to \Pants_n^+$ defined in Remark \ref{rem:surface_pants_graph} is coarsely Lipschitz, there exists a bounded connected subgraph $\mathcal{S}'$ of $\Pants^+_n$ that contains all such pants decompositions $\{ P_{ij} \mid 1 \leq i < j \leq n \}$.

    Define $\mathcal{S}$ to be the $1$--neighborhood of $\mathcal{S}'$ in $\Pants^+_n$, which is also connected and bounded by definition.  
    From the proofs of Lemma \ref{lem:transvection} and Lemma \ref{lem:inversion}, we know that a transvection or inversion preserves some cycle in $\mathcal{S}$.  
    Hence, for each $\phi \in S_n$, we have that $\mathcal{S} \cap \phi.\mathcal{S}$ is nonempty, as $\phi$ fixes or preserves part of $\mathcal{S}$.
\end{proof}

% \textcolor{blue}{We have rearranged the statements of Theorem 5.4, Corollary 5.5, and Corollary 5.6. In addition, we have replaced the paragraph preceding Theorem 5.4 and provided a proof for the revised Theorem 5.4.}

As Remark \ref{rem:orientable_pants_quotients}, the graph $\Pants_n^+ / \Out(F_n)$ is finite and connected.
The quotient of $\Pants_n^+$ by $\Out(F_n)$ is the subgraphs of all pants decompositions of ``unmarked'' orientable surfaces whose fundamental group is isomorphic to $F_n$, which is finite.
Furthermore, this quotient is connected because Elementary move II generates all topological types of orientable surfaces and Proposition \ref{prop:surface_move} implies all pants decompositions of the fixed surface are contained in $\Pants_n^+$.

% Note that there exists a positive number $r = r(n)$ such that for each vertex $P$ of $\Pants^+_n$, there exists a vertex $Q \in B(P, r)$ such that $Q$ is represented by a pants decomposition of an $n$-holed disk.
% On the other hand, the proof of Lemma \ref{lem:connected_S} indicates that there exists an $\Out(F_n)$-invariant connected subgraph $\mathcal{A}$ of $\Pants^+_n$ that contains all pants decompositions of $n$-holed disks.
% In conclusion, all vertices of $\Pants^+_n$ are contained within the $r$-neighborhood of $\mathcal{A}$, which implies the following:

\begin{theorem} \label{thm:orbit_map}
    For every $n \geq 3$, any orbit map $\Out(F_n) \to \Pants^+_n$ is coarsely Lipschitz and coarsely surjective. In particular, one has $\mathcal{K}_n \succ \Pants_n^+$.
\end{theorem}

\begin{proof}
    Let $S_n$ be the set of transvections and inversions with a fixed free basis of $F_n$.
    By Lemma \ref{lem:connected_S}, there exists a bounded connected subgraph $\mathcal{S}$ of $\Pants_n^+$ such that $\phi.\mathcal{S}$ and $\mathcal{S}$ intersect for any transvection or inversion $\phi$.
    Choose $P_0 \in \mathcal{S}$, and let $D$ denote the diameter of $\mathcal{S}$.

    Choose $\psi \in \Out(F_n)$.
    If $\psi = \phi_1 \phi_2 \dots \phi_k$ is a shortest path from the identity, then we have
    \begin{align*}
        d_{\Pants^+}(P_0, \psi.P_0) &= d_{\Pants^+}(P_0, \phi_1 \dots \phi_k.P_0) \\
        &\leq d_{\Pants^+}(P_0, \phi_1.P_0) + \sum_{i = 1}^{k-1} d_{\Pants^+}(\phi_1 \dots \phi_i.P_0, \phi_1\dots\phi_{i+1}.P_0) \\
        &= \sum_{i = 1}^k d_{\Pants^+}(P_0, \phi_i.P_0) \leq kD =D \, d_S(\mathrm{id}, \psi)
    \end{align*}
    where $d_S$ is the word metric of $\Out(F_n)$ with respect to $S_n$.
    Therefore, the orbit map $\Out(F_n) \to \Pants_n^+$ is coarsely Lipschitz.

    Let \(C\) denote the diameter of the quotient graph \(\Pants_n^+ / \Out(F_n)\). 
    Then the \(C\)-neighborhood of the \(\Out(F_n)\)-orbit of \(\mathcal{S}\) covers all vertices of \(\Pants_n^+\). 
    In particular, the orbit map is \(C\)-coarsely surjective.  
    Since \(\mathcal{K}_n\) is quasi-isometric to \((\Out(F_n), d_S)\), it follows that there exists a coarsely Lipschitz and coarsely surjective map from \(\mathcal{K}_n\) to \(\Pants_n^+\). 
    Hence we obtain $\mathcal{K}_n \succ \Pants_n^+$.
\end{proof}

Note that \(\mathcal{K}_n\) is connected, as shown by Culler and Vogtmann \cite[Proposition~3.1.2]{MR830040}. 
Combining this with Theorem~\ref{thm:orbit_map}, we deduce that \(\Pants_n^+\) is connected. 
Moreover, the inclusion \(\Pants_n^+ \hookrightarrow \Pants_n\) is coarsely surjective. 
Hence, we arrive at the following result.

\begin{corollary} \label{cor:lip}
    For every $n \geq 3$, $\Pants_n$ and $\Pants^+_n$ is connected.
\end{corollary}

\bibliographystyle{amsalpha}
\bibliography{references}

\end{document}